\documentclass[reqno,a4paper]{amsart}
\usepackage[utf8]{inputenc}
\usepackage[T1]{fontenc}
\usepackage[english]{babel}
\usepackage{lmodern}
\usepackage{microtype}
\usepackage{csquotes}

\usepackage[backend=bibtex,bibencoding=ascii,giveninits=true,isbn=false,maxbibnames=99]{biblatex}
\renewbibmacro{in:}{\ifentrytype{article}{}{\printtext{\bibstring{in}\intitlepunct}}}

\usepackage{amssymb}

\usepackage{color}
\usepackage{dsfont}
\usepackage{enumerate}
\usepackage{longtable}
\usepackage{stackrel}
\usepackage{mathtools}
\mathtoolsset{centercolon}

\allowdisplaybreaks[4]

\DeclareMathAlphabet{\mathup}{OT1}{\familydefault}{m}{n}

\newtheorem{theorem}{Theorem}[section]
\newtheorem{corollary}[theorem]{Corollary}
\newtheorem{lemma}[theorem]{Lemma}
\newtheorem{proposition}[theorem]{Proposition}

\newtheorem{assumption}[theorem]{Assumption}
\theoremstyle{definition}
\newtheorem{definition}[theorem]{Definition}

\theoremstyle{remark}
\newtheorem{remark}[theorem]{Remark}

\numberwithin{equation}{section}

\newcommand{\dx}[1]{\mathop{}\!\mathup{d} #1}
\newcommand{\pderiv}[3][]{\frac{\mathop{}\!\mathup{d}^{#1} #2}{\mathop{}\!\mathup{d} #3^{#1}}}

\DeclarePairedDelimiter{\abs}{\lvert}{\rvert}

\DeclarePairedDelimiter{\bra}{(}{)}
\DeclarePairedDelimiter{\pra}{[}{]}
\DeclarePairedDelimiter{\set}{\{}{\}}

\newcommand{\eps}{\varepsilon}
\newcommand{\N}{\mathds{N}}
\newcommand{\R}{\mathds{R}}

\newcommand{\cA}{\mathcal{A}}

\newcommand{\cD}{\mathcal{D}}

\newcommand{\cF}{\mathcal{F}}

\newcommand{\cH}{\mathcal{H}}

\newcommand{\cJ}{\mathcal{J}}
\newcommand{\cK}{\mathcal{K}}

\newcommand{\cM}{\mathcal{M}}
\newcommand{\cN}{\mathcal{N}}

\newcommand{\cS}{\mathcal{S}}
\newcommand{\cT}{\mathcal{T}}

\newcommand{\cCE}{\mathcal{CE}}
\DeclareMathOperator{\CE}{CE}
\DeclareMathOperator{\EED}{EED}
\DeclareMathOperator{\Ent}{Ent}
\DeclareMathOperator{\I}{I}
\DeclareMathOperator{\II}{II}

\DeclareMathOperator{\LSI}{LSI}
\DeclareMathOperator{\LSW}{LSW}
\DeclareMathOperator{\mac}{mac}
\DeclareMathOperator{\mic}{mic}

\DeclareMathOperator{\Span}{span}

\usepackage{hyperref}

\DeclareRobustCommand{\SkipTocEntry}[5]{}

\newcommand{\titlestr}{Macroscopic limit of the Becker--Döring equation via gradient flows}
\definecolor{darkblue}{rgb}{0,0,0.6}
\hypersetup{
    pdftitle={\titlestr},    
    pdfauthor={André Schlichting},
     pdfsubject={MSC: Primary 49J40; secondary 35L65, 49J45, 49K15, 60J27, 82C26.},   
    colorlinks=true,       
    linkcolor=darkblue,          
    citecolor=darkblue,        
    filecolor=darkblue,      
    urlcolor=darkblue           
}

\title{\titlestr}
\author{André Schlichting}
\address{Institut für Angewandte Mathematik.
Universität Bonn.}
\email{Schlichting@iam.uni-bonn.de}
\thanks{\today}

\begin{document}

\begin{abstract}
\noindent
This work considers gradient structures for the Becker--Döring equation and its macroscopic limits.
The result of Niethammer~\cite{Niethammer2003} is extended to prove the convergence not only for solutions of the Becker--Döring equation towards the Lifshitz--Slyozov--Wagner equation of coarsening, but also the convergence of the associated gradient structures.
We establish the gradient structure of the nonlocal coarsening equation rigorously and show continuous dependence on the initial data within this framework.
Further, on the considered time scale the small cluster distribution of the Becker--Döring equation follows a quasistationary distribution dictated by the monomer concentration.
\end{abstract}



\maketitle


\section{Introduction}

\subsection{The Becker--Döring model}\label{s:Intro:BD}
In this work, we are interested in gradient structures for the Becker--Döring equation and its macroscopic limits.
The Becker--Döring equation~\cite{Becker1935} is a model for the coagulation and fragmentation of clusters consisting of identical monomers.
The main modeling assumption is only monomers are able to coagulate and fragment with other clusters in a way that the total density of monomers is conserved
\begin{equation}\label{e:BD:massconserve}
 \sum_{l=1}^\infty l n_l(t) = \sum_{l=1}^\infty l n_l(0) =: \varrho_0  \qquad\text{for all } t>0  .
\end{equation}
Hereby, $n_l(t)$ is the density of clusters of size $l$ at time $t$.
The evolution of the densities $n_l(t)$ is given by an countable number of ordinary differential equations of the form
\begin{align}\label{e:BD1}
 \dot n_l(t) = J_{l-1}(t) - J_l(t) \qquad l = 2, 3 \dots
\end{align}
where $J_l$ is the flux from clusters of size $l$ to clusters of size $l+1$.
The system~\eqref{e:BD1} gets closed  with an equation for $n_1$
\begin{equation}\label{e:BD2}
 \dot n_1(t) = -\sum_{l=1}^\infty J_l(t) - J_1(t) =: J_0(t) - J_1(t) ,
\end{equation}
which is chosen, such that formally~\eqref{e:BD:massconserve} is satisfied.
The fluxes $J_l$ are given by mass-action kinetics, that is the rate of coagulation is determined by $a_l n_1 n_l$ and the rate of fragmentation is given by $b_{l+1} n_{l+1}$, where $a_l$ and $b_{l+1}$ are rate factors only depending on $l$.
This leads to the constitutive relation
\begin{equation}\label{e:BD:flux}
 J_l(t) = a_l n_1(t) n_l(t) - b_{l+1} n_{l+1}(t) , \qquad l=1,2,\dots .
\end{equation}
The detailed balance condition for this system reads $J_l(t) = 0$ for all $l$, satisfied by a one-parameter family of equilibrium solutions
\begin{equation}\label{e:BD:equilibrium}
 \omega_l(z) := z^l Q_l, \qquad\text{with}\qquad Q_1 := 1  \qquad\text{and}\qquad  Q_l := \prod_{j=1}^{l-1} \frac{a_j}{b_{j+1}} .
\end{equation}
To specify the long-time behavior, we introduce the convergence radius of the series $z \mapsto \sum_l l z^l Q_l$ by $z_s \in [0,\infty]$ as well as its value at the convergence radius
\begin{equation}\label{e:BD:critical_mass}
\varrho_s := \sum_{l=1}^\infty l z_s^l Q_l \in [0, \infty].
\end{equation}
We are interested in the regime where $z_s \in (0,\infty)$ and $\varrho_s \in (0,\infty)$.
We will assume that the rates are explicitly given as follows:
\begin{assumption}[Rates]\label{ass:BD:rates}
For $\alpha \in [0,1)$, $\gamma \in (0,1)$ and $z_s, q >0$ define the coagulation and fragmentation rate of a monomer for a cluster of size $l$ by
\begin{equation*}
 a_l := l^\alpha \qquad\text{and}\qquad b_l := l^\alpha \bra*{z_s + q l^{-\gamma}} .
\end{equation*}
\end{assumption}
Hereby, the parameter $z_s$ is consistent with its definition as radius of convergence (cf.\ Lemma~\ref{lem:BD:assymptotic_Q}) and $\varrho_s$ as defined in~\eqref{e:BD:critical_mass} is strictly positive and finite under Assumption~\ref{ass:BD:rates}.

Then, as investigated by~\cite{Ball1986} solutions to the Becker--Döring equation with $\varrho_0 \leq \varrho_s$ converge to the equilibrium state $\omega_l(z)$, where $z=z(\varrho_0)$ is given such that $\sum_{l=1}^\infty l z^l Q_l = \varrho_0$ and the convergence takes place in a weighted $\ell^1$ space
\begin{equation*}
  \lim_{t\to \infty} \sum_{l=1}^\infty l \; \abs*{ n_l(t) - \omega_l(z)} = 0
\end{equation*}
In the case $\varrho_0 > \varrho_s$, it holds
\begin{equation*}
  \lim_{t\to\infty} n_l(t) = \omega_l(z_s) \qquad\text{for each $l\geq 1$}.
\end{equation*}
Hence, the excess mass $\varrho_0-\varrho_s>0$ vanishes in the limit $t \to \infty$.
The interpretation is, that the excess mass is contained in larger and larger clusters as times evolve.
These large clusters form a new phase, e.g.\ liquid droplets formed out of supersaturated vapor.
It is the aim of the is work to add some aspect to the understanding of the formation of the new phase.

The crucial ingredient for the above convergence statements is the existence of a Lyapunov functional $\cF$ of the form of a relative entropy $\cF_z(n) := \cH(n \mid \omega(z))$. Hereby, $z>0$ is a parameter selecting the stationary state and the relative entropy is defined by
\begin{align}\label{e:def:Lyapunov}
  \cH(n \mid \omega) := \sum_{l=1}^\infty \omega_l \psi\bra*{\frac{n_l}{\omega_l}} \quad\text{with}\quad \psi(a) := a\log a - a +1 , \text{ for } a > 0 .
\end{align}
A calculation shows that it is formally decreasing along solutions to the Becker--Döring equation
\begin{equation}\label{e:BD:Dissipation}
  \pderiv{\cF_z(n(t))}{t} = - \sum_{l=1}^\infty \bra*{ a_l n_1 n_l - b_{l+1} n_{l+1}} \bra*{\log a_l n_1 n_l - \log b_{l+1} n_{l+1}} =: - \cD(n(t)) \leq 0 .
\end{equation}
Hence, the Lyapunov function can be interpreted as a free energy dissipating along the flow.
This indicates, that the free energy is minimized as $t\to \infty$.
By the mass conservation~\eqref{e:BD:massconserve}, we expect the long-time limit to be the solution to the following minimization problem
\begin{equation}\label{e:BD:FEmin}
  \inf\set[\bigg]{ \cF(n) : \sum_{l=1}^\infty l n_l = \varrho_0 } =
  \begin{cases}
    \cF_z(\omega(z)) , &\varrho_0 \leq \varrho_s ; \\
    \cF_{z_s}(\omega(z_s)) , &\varrho_0 > \varrho_s .
  \end{cases}
\end{equation}
In the first case the infimum is attained and the parameter $z = z(\varrho_0)$ is chosen such that $\sum_{l=1}^\infty l \omega_l(z) = \varrho_0$.
In the second case the infimum is not attained (cf.\ \cite[Theorem 4.4]{Ball1986}). From now on, we choose $z = z(\varrho_0)$ in this particular form and omit the supscript.
Hence, the functional reflects correctly the long-time behavior of the equation.
Moreover, the Lyapunov function has the form of a relative entropy and the question arises, whether their exists a gradient structure for the Becker--Döring equation having this relative entropy as driving free energy.

\subsection{Gradient flow structure}\label{s:Intro:BD:GF}
To bring the system into the framework of gradient-flows, it is helpful to interpret the Becker--Döring equation as the following system of chemical reactions
\begin{equation}\label{e:BD:ChemReact}
  X_1 + X_{l-1} \stackrel[b_{l}]{a_{l-1}}{\rightleftharpoons} X_{l}, \qquad  l = 2 , 3, \dots\ .
\end{equation}
Hereby, $X_l$ denotes a cluster of size $l$ and the rates for coagulation $\set*{a_l}_{l\geq 1}$ and fragmentation $\set*{b_l}_{l\geq 2}$ are positive as in Assumption~\ref{ass:BD:rates}.
In this formulation, we can use the  gradient structure as observed by Mielke~\cite{Mielke2011a} for chemical reactions under detailed balance condition and it turns out that the Becker--Döring equation is indeed a gradient flow with respect to the Lyapunov function~\eqref{e:def:Lyapunov} under a suitable metric. The same metric was discovered by Maas~\cite{Maas2011} in the setting of reversible Markov chains.

The existence of the metric depends crucially on the detailed balance condition satisfied by the equilibrium~\eqref{e:BD:equilibrium}
\begin{equation}\label{e:BD:DBC}
  a_l \omega_1 \omega_l = b_{l+1} \omega_{l+1} =: k^l,
\end{equation}
where $k^l$ is the stationary equilibrium flux and the implicit parameter $z$ is chosen according to $\varrho_0$ as described after~\eqref{e:BD:FEmin}. The equations~\eqref{e:BD1}, \eqref{e:BD2}, \eqref{e:BD:flux} can be compactly rewritten with the help of~\eqref{e:BD:DBC} as
\begin{equation}\label{e:BD:ChemMaster}
  \dot n = - \sum_{l=1}^\infty k^l \bra*{ \frac{n_1 n_l}{\omega_1 \omega_l} - \frac{n_{l+1}}{\omega_{l+1}}} \bra*{ e^1 + e^l - e^{l+1}} ,
\end{equation}
with $e^l_i=0$ for $i\ne l$ and $e^l_l=1$ for $l\in \N$.
Since the free energy $\cF$ is of the form of a relative entropy~\eqref{e:def:Lyapunov}, we can identity its variation as
\begin{equation*}
  D\cF(n) = \bra*{ \log \frac{n_1}{\omega_1}, \dots, \log \frac{n_i}{\omega_i}, \dots} .
\end{equation*}
Then, the gradient flow formulation of the Becker-Döring equation takes the form
\begin{equation}\label{e:BD:GF}
  \dot n = - \cK(n) \, D\cF(n) ,
\end{equation}
where the Onsager matrix $\cK$ is defined by
\begin{equation}\label{e:BD:Onsager}
  \cK(n) := \sum_{l=1}^\infty k^l \ \Lambda\bra*{ \frac{n_1 n_l}{\omega_1 \omega_l} , \frac{n_{l+1}}{\omega_{l+1}}} \ \bra*{ e^1 + e^l- e^{l+1}} \otimes \bra*{ e^1 + e^l- e^{l+1}}
\end{equation}
and $\Lambda(\cdot,\cdot)$ is the logarithmic mean given for $a,b>0$ by
\begin{equation}\label{e:def:LogMean}
 \Lambda(a,b) = \int_0^1 a^s b^{1-s} \dx{s} =
\begin{cases}
 \frac{a-b}{\log a - \log b } &, a\ne b  \\
 a &, a=b .
\end{cases}
\end{equation}
The identification of~\eqref{e:BD:ChemMaster} and~\eqref{e:BD:GF} is based on the algebraic identity
\[
  \Lambda\bra*{a b, c} \ \bra[\big]{ \log a + \log b - \log c} = a b - c \qquad\text{ for }\qquad a,b,c>0 .
\]
We refer to Appendix~\ref{s:GScfModels} for the more general structure behind this identities and applications to other coagulation and fragmentation models.

\subsection{Variational characterization}\label{s:Intro:BD:Variational}
The gradient flow formulation allows for a variational characterization initiated by de Giorgi and its collaborators~\cite{dGMT1980} under the name of curves of maximal slope.
From the interpretation of the Becker--Döring model as chemical reaction, it is clear the the total number of particles is conserved, which suggests to define the state manifold
\begin{equation*}
  \cM := \set*{ n \in \R_+^\N :  \sum_{l=1}^\infty l n_l = \varrho_0 } .
\end{equation*}
Possible variations of the state manifold consistent with the Becker-Döring dynamic are given by the linear space $\cT\cM = \operatorname{span}\set*{ e^1 + e^l - e^{l+1} : l \in \N}$. By the definition of the Onsager matrix~\eqref{e:BD:Onsager}, we have that the following space is well-defined
\begin{equation}\label{e:def:cotangent}
 \cT_n^*\cM := \set*{\phi\in \R^\N : \exists s\in \cT\cM \text{ such that }  s=
  -\cK(n) \phi} .
\end{equation}
A crucial ingredient to study the underlying metric structure is the continuity equation and curves of finite action.
\begin{definition}[Curves of finite action]\label{def:BD:CurvesFiniteAction}
 A pair $[0,T]\ni t\mapsto (n(t),\phi(t))\in \cM \times \cT^*_{n(t)}\cM$ is a solution to the continuity equation, denoted by $(n,\phi)\in \cCE_T$, if it satisfies
\begin{enumerate}[ (i) ]
 \item $n(\cdot): [0,T]\to \cM$ is absolute continuous.
 \item The pair $(n,\phi)$ satisfies the continuity equation
    for $t\in (0,T)$ in the weak form, that is for all $\psi\in
    C^1_c((0,T),\R)$ and all $l\in \N$ holds
    \begin{equation}\label{e:cCE}
        \int_0^T\Big( \dot\psi(t)\; {n}_{l}(t)- \psi(t)\;\big(\cK(n(t)) \phi(t)\big)_l \Big) \dx{t} = 0 .
    \end{equation}
\end{enumerate}
The \emph{action} $\cA$ of a pair $(n,\phi) \in \cM\times \cT^*_{n}\cM$ is defined by
\begin{equation}\label{e:BD:action}
 \cA(n,\phi) := \phi(t) \cdot \cK(n(t)\phi(t) = \sum_{r=1}^R k^r \hat n^\omega_r \abs*{ \nabla_r \phi(t)}^2 ,
\end{equation}
where $\nabla_r \phi := \phi_{r+1} - \phi_r - \phi_1$ and
\begin{equation}\label{e:def:hatn}
  \hat n_l^\omega := \Lambda\bra*{\frac{n_1 n_l}{\omega_1 \omega_l} , \frac{n_{l+1}}{\omega_{l+1}}} .
\end{equation}
A curve $(n,\phi)\in \cCE_T$ is called a \emph{curve of finite action}, if
\begin{equation*}
 \sup_{t\in [0,T]} \cF(n(t)) <\infty, \quad  \int_0^T \cA(n(t), \phi(t)) \dx{t} < \infty ,\quad\text{and }\quad \int_0^T \cD(n(t)) \dx{t} < \infty ,
\end{equation*}
where $\cD(n) = \cA(n,-D\cF(n))$ is given as in~\eqref{e:BD:Dissipation}.
\end{definition}
The nonlocal gradient $\nabla_r \phi$ in~\eqref{e:BD:action} can be avoided by interpreting the monomer concentration $n_1 = \varrho - \sum_{l=2}^\infty l n_l$ as a nonlocal boundary condition. Along this idea a Fokker-Planck equation with such type of boundary condition having similar features like the Becker--Döring model was recently introduced in~\cite{ConSch17}.

Curves of finite action give a variational formulation to solutions of the Becker--Döring equation.
In comparison to the direct gradient flow equation~\eqref{e:BD:GF}, this avoids regularity questions arising from the application of the chain rule.
The concept was introduced in~\cite{dGMT1980} and further investigated in~\cite{AGS05}: For any curve $(n,\phi)\in \cCE_T$ of finite action holds
\begin{equation}\label{e:BD:Jintro}
 \cJ(n) := \cF(n(T)) - \cF(n(0)) + \frac{1}{2} \int_0^T \cD(n(t)) \dx{t} + \frac{1}{2} \int_0^T \cA(n(t), \phi(t)) \dx{t} \geq 0 .
\end{equation}
Moreover, equality is attained if and only if $n$ is a solution of~\eqref{e:BD:ChemMaster}.

We provide the crucial observation of the proof, which follows formally by evaluating
\begin{align*}
  \pderiv{}{t} \cF(n(t)) &=  D\cF(n(t)) \cdot \dot n(t) \stackrel{\eqref{e:cCE}}{=} D\cF(n(t)) \cdot \cK(n(t)) \phi(t) \\
  &\geq - \frac{1}{2} D\cF(n(t)) \cdot \cK(n(t))D\cF(n(t)) - \frac{1}{2} \phi(t) \cdot \cK(n(t)) \phi(t) ,
\end{align*}
where we used that $\cK$ is positive semidefinite and the Cauchy--Schwarz inequality.
The equality case is read off from the equality case in Cauchy--Schwarz.
For a rigorous treatment in a similar situation, we refer to~\cite[Section 2.5]{EFLS16}.
We use this variational structure to pass to the limit after a suitable rescaling.

\subsection{The macroscopic limit}\label{s:Intro:MacLim}
The connection between the Becker--Döring equation with positive excess mass $\varrho_0 - \varrho_s = \bar\varrho >0$ and a macroscopic theory of coarsening is due to Penrose~\cite{Penrose1997}.
He observed by formal asymptotics that the macroscopic part of the Becker-Döring dynamics converges after a suitable rescaling (cf.\ Section~\ref{s:scaling}) to a classical coarsening model introduced by Lifshitz and Slyozov~\cite{Lifshitz1961}, and Wagner~\cite{Wagner1961}
\begin{equation}\begin{split}\label{e:LSW}
  \partial_t \nu_t + \partial_{\lambda}\big( \lambda^\alpha( u(t) - q \lambda^{-\gamma})\nu_t\big) = 0 \quad \text{with}\quad  u(\nu_t) = \frac{q \int \lambda^{\alpha-\gamma} \dx\nu_t}{\int \lambda^\alpha \dx\nu_t} .
\end{split}\end{equation}
Hereby, the measure $\nu_t(\dx{\lambda})$ is the distribution of particles of macroscopic size $\lambda\in \R_+$.
Moreover, the parameters $\alpha$, $\gamma$ and $q$ satisfy Assumption~\ref{ass:BD:rates} and we will call the nonlocal conservation law \eqref{e:LSW} the LSW equation in the following. Formally, the total mass is conserved and the evolution stays in the state manifold for any $t>0$
\begin{equation*}
 \nu_t \in M = \set*{\nu \in C_c^0(\R_+)^* \;\middle|\; \int \lambda \;\nu(\dx\lambda) = \bar\varrho } .
\end{equation*}
The LSW equation are a gradient flow as formally observed by Niethammer \cite[Section 4]{Niethammer2004}.
The driving energy of the system is given exactly by the first oder expansion of the macroscopic part of a suitable rescaling of the free energy \eqref{e:def:Lyapunov} (with $z=z_s$) driving the Becker--Döring equation (cf.\ Lemma~\ref{lem:BD:expansionF})
\begin{equation}\label{e:LSW:energy}
  E(\nu) := \frac{q}{1-\gamma} \int \lambda^{1-\gamma} \; \nu(\dx\lambda) .
\end{equation}
Let us introduce a formal Riemannian structure and define a tangent space on $M$ by $T_\nu M := \set*{s : \R_+ \to \R \;\middle|\; \int \lambda s \dx{\lambda} = 0 }$.
An identification of tangent and cotangent vectors is obtained via the operator $K(\nu): T_{\nu}^* M \to T_\nu M$ given by
\begin{align}\label{e:LSW:Onsager}
  K(\nu) w := - \partial_{\lambda}\bra*{\lambda^{\alpha} w \nu} \quad \text{for} \quad w\in T_{\nu}^* M &:=  \set*{w: \R_+ \to \R \;\middle|\;  \exists s \in T_\nu M : K(\nu) v= s } .
\end{align}
By an integration by parts of the identity $0 = \int \lambda s \, \nu(\dx\lambda)$ holds the inclusion property $T_{\nu}^* M \subseteq \set*{w: \R_+ \to \R \;\middle|\; \int \lambda^{\alpha} w \; \nu(\dx{\lambda}) = 0}$.
Let us formally derive the gradient structure for the LSW equation (cf.\ \cite[Section 4]{Niethammer2004}), that is we assume all differentials and quantities to be smooth enough.
The differential of the energy~\eqref{e:LSW:energy} is given for some $s\in T_{\nu}M$ by using the identification $s=-K(\nu) w$ with $w\in T^*_\nu M$
\begin{align*}
 DE(\nu)\cdot s &= \frac{q}{1-\gamma} \int \lambda^{1-\gamma} s \dx{\lambda} =- \int \bra*{u\lambda - \frac{q}{1-\gamma} \lambda^{1-\gamma}} s \, \dx{\lambda}\\
 &= - \int \lambda^{\alpha} \bra{ u  - q \lambda^{-\gamma}} w \, \nu(\dx{\lambda}) ,
\end{align*}
where $u \in \R$ can be chosen such that $u- DE(\nu) \in T_{\nu}^* M$ thanks to $\int \lambda^\alpha w \nu(\dx\lambda)=0$.
Then, the gradient flow in weak form satisfies for all $\tilde s \in T_\nu M$
\begin{equation*}
 \int \lambda^\alpha w \tilde w \; \nu(\dx{\lambda}) =g_{\alpha,\nu}(\partial_t \nu,\tilde s) = - DE(\nu) \cdot \tilde s =  \int \lambda^{\alpha} \bra{ u  - q \lambda^{-\gamma}}\tilde w\; \nu(\dx{\lambda}) ,
\end{equation*}
where $\partial_t \nu = -\partial_\lambda(\lambda^{\alpha} w \nu)$ and $\tilde s = - \partial_{\lambda}\bra*{\lambda^\alpha \tilde w \nu}$ in distribution.
Hence, we obtain the identification
\begin{equation*}
 w   =  u  - q \lambda^{-\gamma},
\end{equation*}
where $u=u(\nu)$ is a Lagrangian multiplier chosen such that $w\in T_\nu^*M$, that is it satisfies the constraint $\int \lambda^\alpha w \, \nu(\dx\lambda) = 0$ and is formally given by \eqref{e:LSW}.
Hence, the gradient flow of the energy $E$ with respect to the metric induced by $K$ is given by
\begin{equation*}
 \partial_t \nu_t = - K(\nu) D E(\nu) = - \partial_\lambda\bra*{\lambda^\alpha\bra*{u(\nu_t) - q \lambda^{-\gamma}}\nu_t} ,
\end{equation*}
where $u(\nu_t)$ given by~\eqref{e:LSW}.
To make the above observation rigorous, we use the de Giorgi formalism of curves of maximal slope.
Up to technical details, which is dealt with in Section~\ref{S:LSW}, we can define an action functional as follows: For a pair $(\nu,w)$ solving the continuity equation $\partial_t \nu_t + \partial_{\lambda}\bra*{ \lambda^\alpha w_t \nu_t }  = 0$ in distributions, denoted by $(\nu,w)\in \CE_T$, the action is defined by
\begin{equation*}
   A(\nu_t,w_t) := \int \lambda^\alpha \abs*{w_t}^2 \dx\nu_t .
\end{equation*}
Then, by the identification of tangent and co-tangent vectors via $s= - \partial_\lambda(\lambda^\alpha w \nu)$, we obtain that the dissipation is given by
\begin{equation}\label{e:LSW:dissipation}
  D(\nu_t)  = A\bra*{\nu_t, u(t)  -DE(\nu_t)} = \int \lambda^\alpha \abs*{ u(\nu_t) - q \lambda^{-\gamma}}^2 \dx\nu_t ,
\end{equation}
where $u(\nu_t)\in L^2((0,T))$ given by~\eqref{e:LSW} ensures that $u(\nu_t) - q \lambda^{-\gamma}\in T_\nu^* M$, i.e.~it is a valid cotangent vector satisfying $\int \lambda^\alpha \bra*{u(\nu_t) - q\lambda^{-\gamma}} \dx{\lambda} =0$.

The functional $J(\nu)$, which completely characterizes solutions to~\eqref{e:LSW} (cf.\ Proposition~\ref{prop:LSW:DeGiorgi}) is defined by
\begin{equation}\label{e:LSW:DeGiorgi}
  J(\nu) := E(\nu_T) - E(\nu_0) + \frac{1}{2}\int D(\nu_t) \dx{t} + \frac{1}{2} \int A(\nu_t, w_t) \dx{t} \geq 0 ,
\end{equation}
with $J(\nu) = 0 $ if and only if $\nu_t$ is a weak solution to the LSW equation~\eqref{e:LSW}.

The main application of this variational framework is to prove the convergence of the Becker-Döring gradient structure to the LSW gradient structure.
In addition, the variational characterization of the LSW equation together with a compactness statement for curves of finite action (cf.\ Proposition~\ref{prop:LSW:compactness}) allows to proof continuous dependence on the initial data (cf.\ Corollary~\ref{cor:LSW:continitial}).

\subsection{Passage to the limit}
The macroscopic limit is rigorously derived by Niethammer~\cite{Niethammer2003}.
There, the main technical tool was to pass to the limit in the energy-dissipation relation associated with the rescaled Becker--Döring equation to obtain the energy-dissipation relation of the LSW equation. The one for solutions to the Becker--Döring equation is obtained by integrating the identity~\eqref{e:BD:Dissipation} in time
\begin{equation}\label{e:BD:EED}
  \cF(n(T)) -  \cF(n(0)) + \int_0^T \cD(n(t)) \dx{t} = 0  .
\end{equation}
The functional $\cJ$ from~\eqref{e:BD:Jintro} contains the identity~\eqref{e:BD:EED}, since for solutions of the Becker--Döring equation it holds $\cA\bra*{n(t),-D\cF(n(t))} = \cD(n(t))$.

Likewise, from~\eqref{e:LSW:DeGiorgi} and~\eqref{e:LSW:dissipation} follows that the LSW equation satisfy the energy-dissipation identity
\begin{equation*}
  E(\nu_T) - E(\nu_0) + \int_0^T D(\nu_t) \dx{t} = 0 .
\end{equation*}
where $u(t)$ is given in~\eqref{e:LSW}.

The contribution of this work is to lift the convergence statement from the level of energy-dissipation relations along solutions to the functionals $\cJ$ and $J$ along curves of finite action. Hereby, by doing so no essential new technical difficulties arrise, which underlines the fact that the gradient structure is natural for these types of equations.
We prove that a suitable rescaling of the functional $\cJ^\eps$ converges to the functional $J$ in an evolutionary $\Gamma$-convergence sense under the assumption of well-prepared initial data (see Theorem~\ref{thm:Scale:conv:CfA}).
In particular, the gradient structure of the Becker--Döring equation converges to the one of the LSW equation (cf.\ Theorem~\ref{thm:Scale:conv:CfA}) and in particular it implies the convergence of solutions (cf.\ Corollary~\ref{cor:ConvDeGiorgi}).
This program follows the ideas of Sandier and Serfaty \cite{Sandier2004}, and was later generalized by Serfaty \cite{Serfaty2011}.

The ingredients of the proof of convergence are based on: \emph{(i)} the variational characterization of the Becker-Döring equations in Section~\ref{s:Intro:BD:Variational}, which follows the gradient structure established by~\cite{Mielke2011a}; \emph{(ii)} the rigorous variational characterization of solutions to the LSW equation in Section~\ref{S:LSW}, which extends the formal gradient structure of~\cite[Section 4]{Niethammer2004}; \emph{(iii)} a priori estimates for the variational framework of the Becker-Döring gradient structure in Section~\ref{s:AprioriBD}, which lifts many of the results of~\cite{Niethammer2005b} from solutions of the Becker-Döring system to curves of finite action.

Another motivation to reconsider the proof of~\cite{Niethammer2003} is that systems possessing a gradient structure can be well described by studying convexity properties of the free energy with respect to the implied metric. Especially, the results of~\cite{Penrose1989} suggest, that the system shows dynamic metastability as described by~\cite{Otto2007} for gradient systems. Under this point of view also the additional results on quasistationarity in the next subsection are first steps towards a characterization of dynamic metastability of the Becker-Döring equations. 

\subsection{Well preparedness of initial data and quasistationarity}
A crucial assumption in the approach of showing convergence via curves of maximal slope is the well preparedness of initial data, which assumes that the rescaled free energy of the Becker--Döring gradient structure converges to the one of the LSW gradient structure
\begin{equation*}
  \cF^\eps(n^\eps(0)) \to E(\nu_0) \qquad\text{as}\qquad \eps \to 0 .
\end{equation*}
The second contribution of this work is to show that on the rescaled time-scale, the Becker--Döring equation reach instantaneously a quasistationary equilibrium, which is dictated only by the monomer concentration.
On the other hand, the monomer concentration follows closely a macroscopic quantity similarly defined as $u$ in \eqref{e:LSW}.
The crucial ingredient in the proof is an energy-dissipation estimate based on a logarithmic Sobolev inequality similarly to the one used in~\cite{Canizo2015} to proof convergence to equilibrium in the noncondensing case $\varrho_0 \leq \varrho_s$.

The quasistationary result shows, that the microscopic part of the rescaled free energy $\cF^\eps(n^\eps(t))$ vanishes for almost every $t\geq 0$. It does so by proving a separation of time scales. The fast scale is the relaxation time of small clusters towards a local equilibrium, which can be understood as the response to the slower coarsening time of the large clusters. 
On the level of conergence of gradient flows, this is a step towards showing, that only the macroscopic part of the rescaled free energy $\cF^\eps(n^\eps(0))$ has to convergence towards $E(\nu_0)$ to ensure well prepared initial date.
The conjecture is, that the microscopic part is automatically well prepared on the observed rescaled time-scale.
This is consistent with the continuous dependence on the initial data of the LSW equation, which is valid under the assumption of convergence of the macroscopic energy for the initial data (see Corollary~\ref{cor:LSW:continitial}).

\addtocontents{toc}{\SkipTocEntry}
\subsection*{Outline}

The next Section~\ref{S:main} contains in Section \ref{s:scaling} the rescaling of the Becker-Döring gradient flow structure. This enables us to state the main results in Section~\ref{S:main:main}.
In Section \ref{S:LSW}, we prove the gradient flow structure of the LSW equation and prove the continuous dependence on the initial data within this framework.
Section \ref{S:Limit} contains some a priori estimate for the Becker--Döring system in Section~\ref{s:AprioriBD}, which allow then to the limit in the gradient structure in Section~\ref{S:Limit:sub} and finally we prove the quasistationary equilibrium of the small clusters in Section~\ref{S:quasi}.
We conclude the paper with an Appendix~\ref{s:GScfModels} showing that also more general discrete coagulation and fragmentation models fall into this framework.
Moreover, another Appendix~\ref{s:assymptotic_Q} provides an elementary estimate.

\section{Main results}\label{S:main}

\subsection{Heuristics and scaling}\label{s:scaling}

From now, we consider the Becker--Döring system with initial total mass $\varrho_0>\varrho_s$ and rates satisfying~Assumption~\ref{ass:BD:rates}.
Moreover, the reference state for the free energy is given by $\omega = \omega(z_s)$ as defined in~\eqref{e:BD:equilibrium}.

We fix a scale $\eps^{-1}$ of the large cluster for some $\eps >0$ and consider the first order expansion of the energy in $\eps$. For some cut-off $l_0$, we introduce for $l\geq l_0$ the rescaled variable $\lambda=\eps l$ and treat $\lambda$ as continuous variable on $\R_+$.

We rescale the cluster density $n_l$ by $\eps^2$ and define the empirical measure by
\begin{equation}\label{e:def:BD:mac}
 \nu^\eps(\dx\lambda) := \bra*{\Pi_{\mac}^\eps n}(\dx{\lambda}) := \eps \sum_{l\geq l_0} \delta_{\eps l}(\dx\lambda) \frac{n_l}{\eps^2} = \frac{1}{\eps} \sum_{l\geq l_0} \delta_{\eps l}(\dx\lambda) n_l \ .
\end{equation}
That is for each $\zeta\in C_c^0(\R)$ holds
\begin{equation*}
  \int_0^\infty \zeta(\lambda)\, \nu^\eps(\dx{\lambda}) = \frac{1}{\eps} \sum_{l\geq l_0} \zeta(\eps l) n_l .
\end{equation*}
This scaling preserves the mass in the large cluster, which follows by approximating $\zeta(\lambda)=\lambda$ with cut-off functions.

The leading order contribution of the free energy is given by the free energy of the large clusters $l\geq l_0$.
This part of the free energy~\eqref{e:def:Lyapunov} can be expanded (cf.\ Lemma~\ref{lem:BD:expansionF}) as follows
\begin{align}
 \cF(n)&\geq \sum_{l \geq l_0} \omega_l \psi\bra*{\frac{n_l}{\omega_l}}
  = \bra*{\frac{q}{z_s(1-\gamma)} \sum_{l\geq l_0} l^{1-\gamma} n_l}\bra*{1+O(l_0^{-\sigma})+O(l_0^\gamma w_{l_0})} \nonumber \\
  &=  \frac{\eps^\gamma q}{z_s(1-\gamma)} \int \lambda^{1-\gamma} \dx\nu^\eps \bra*{1+O(l_0^{-\sigma})+O(l_0^\gamma w_{l_0})}, \label{e:Scaling:1stOrderEnergy}
\end{align}
for some $\sigma >0$.
To match the macroscopic energy~\eqref{e:LSW:energy}, we define the rescaled free energy as
\begin{equation*}
 \cF^\eps(n) =  \frac{z_s}{\eps^{\gamma}} \cF(n) .
\end{equation*}
The main result of~\cite{Ball1986} states that the total free energy decreases to zero as $t\to \infty$. Hence, one possible way to obtain initial data $n^\eps(0)$ with $\cF^\eps(n^\eps(0)) = O(1)$ is to introduce a time $t_\eps$ such that $\cF(n(t_\eps)) = O(\eps^\gamma)$ and set $n^\eps(0)= n(t_\eps)$. In particular, this implies by the results of \cite{Penrose1989}, that for $\eps$ small enough, all possible existing metastable states are already broken down.

The expansion \eqref{e:Scaling:1stOrderEnergy} also shows, that the cut-off $l_0$ has to satisfy two conditions (cf.\ \eqref{e:BD:expansionF} and~\eqref{e:Apriori3})
\begin{equation*}
  \lim_{\eps\to 0} l_0^\gamma \omega_{l_0} = 0 \qquad\text{and}\qquad  \lim_{\eps\to 0} \max\set*{ l_0^\gamma , l_0^\alpha} \sqrt{\cF(n^\eps(0))} = 0  .
\end{equation*}
By taking into account the asymptotic of $\set*{Q_l}_{l\geq 1}$ (cf.\ Lemma~\ref{lem:BD:assymptotic_Q}) and recalling $\omega_l = z_s^l Q_l$, the cut-off $l_0$ can be chosen as
\begin{equation}\label{ass:cutoff}
 l_0 := \lfloor \eps^{-x} \rfloor \qquad\text{ for some }\qquad x\in \bra*{0,\tfrac{1}{2}} .
\end{equation}
We consider only states $n$ such that free energy is of order $\eps^\gamma$, that is we consider the restricted state space
\begin{equation*}
 \cM^{\eps} := \set[\bigg]{ n\in \R_+^\N : \sum_{l\geq 1} l n_l = \varrho_0 \ \text{ and }\  \cF(n) \leq \eps^\gamma}.
\end{equation*}
Likewise, the differential of the free energy for states $n^\eps\in \cM^\eps$ will be of order $\eps^\gamma$ and hence covectors will be also on scale $z_s^{-1} \eps^\gamma$, that is we define a rescaled vector field $w^\eps$ by
\begin{equation}\label{def:Scale:vectorfield}
\nabla_l \phi = (e^{l+1} - e^l - e^1) \cdot \phi = \phi_{l+1}- \phi_l - \phi_1 =:z_s^{-1}  \eps^\gamma w^\eps(\eps l) .
\end{equation}
The rescaling of tangent vectors is then determined by the rescaling necessary for obtaining the macroscopic Onsager operator~\eqref{e:LSW:Onsager}.
This follows heuristically by expanding the Onsager matrix~\eqref{e:BD:Onsager}
\begin{equation*}
  (\cK(n) \phi)_{l} = - \eps^{1-\alpha} \partial_{\lambda}^\eps\bra*{ \lambda^\alpha \eps^\gamma w^\eps \, \nu^\eps} \bra*{1+o(1)} \quad\text{with}\quad \partial_\lambda^\eps f(\lambda) := \frac{f(\lambda+\eps)-f(\lambda)}{\eps} .
\end{equation*}
Hence, we define the rescaled Onsager operator by
\begin{equation*}
 (\cK^\eps(n) w^\eps)(\eps l) := \frac{1}{\eps^{1-\alpha+\gamma}} (\cK(n) \phi)(l) ,
\end{equation*}
where $w^\eps$ and $\phi$ are given by the relation~\eqref{def:Scale:vectorfield}.
This rescaling translates to the action $\cA(n,\phi)$~\eqref{e:BD:action} and we define the
rescaled action by
\begin{equation}\label{def:Scale:action}
 \cA^\eps(n,w^\eps) := \frac{z_s}{\eps^{1-\alpha + 2 \gamma}} \sum_{l\geq 1} k^l \hat n^{\omega}_l \abs*{\nabla_l \phi}^2  .
\end{equation}
Since, the dissipation is given as $\cD(n) := \cA(n,-D\cF(n))$, the rescaling is the same and we define $\cD^\eps(n) = z_s \eps^{-\bra{1-\alpha + 2\gamma}} \cD(n)$.
Hence the total rescaling between cotangent and tangent vectors is $\eps^{1-\alpha+\gamma}$, which fixes the time scale for the macroscopic process.

Now, we introduce rescaled curves of finite action in analog to Definition~\ref{def:BD:CurvesFiniteAction}.
By abuse of notation the new time-scale $t/\eps^{1-\alpha+\gamma}$ is still denoted by $t$.
\begin{definition}[Rescaled curves of finite action]
  A weak solution $[0,T]\ni t \mapsto (n^\eps(t), w^\eps(t))$ to the rescaled continuity equation
  \begin{equation*}
     \int_0^T \bra*{ \dot\psi(t) n^\eps_l(t) - \psi(t) \bra*{\cK^\eps(n^\eps(t)) w^\eps(t)}_l } \dx{t} = 0 ,\qquad\text{for all } \psi\in C_c^1((0,T);\R)
  \end{equation*}
  denoted by $(n^\eps,w^\eps)\in \cCE^\eps_T$ is called a rescaled curve of finite action if
  \begin{equation*}
    \sup_{t\in [0,T]} \cF^\eps(\nu^\eps_t) < \infty , \quad
    \int_0^T \cA^\eps(n^\eps(t), w^\eps(t)) \dx{t}  < \infty
    \quad \text{and} \quad
    \int_0^T \cD^\eps(n^\eps(t)) \dx{t} < \infty .
  \end{equation*}
  Moreover, for such a curve we define the rescaled functional characterizing curves of maximal slope by
  \begin{equation}\label{e:BD:DeGiorgi}
    \cJ^\eps(n^\eps) := \cF^\eps(n^\eps(T)) - \cF^\eps(n^\eps(0)) + \frac{1}{2} \int_0^T \!\! \cD^\eps(n^\eps(t)) \dx{t} + \frac{1}{2} \int_0^T \!\! \cA^\eps(n^\eps(t), w^\eps(t)) \dx{t} \geq 0 .
  \end{equation}
  In particular solutions such that $\cJ^\eps(n^\eps) = 0$ satisfy the time-rescaled Becker--Döring equation
  \begin{align}\label{e:BD:GF:rescaled}
    \dot n^\eps(t) = - \eps^{1-\alpha+\gamma} \cK(n^\eps(t)) D\cF(n^\eps(t))  .
  \end{align}
\end{definition}

\subsection{Convergence of the gradient structures}\label{S:main:main}

The functionals $\cJ^\eps$ \eqref{e:BD:DeGiorgi} and $J$~\eqref{e:LSW:DeGiorgi} are used to characterize solutions of the Becker--Döring and LSW equations in a variational way, respectively.
The main idea to show convergence of the Becker--Döring equation to the LSW equation, which goes back to~\cite{Sandier2004} (cf.\ \cite{Serfaty2011}), is to prove $\liminf_{\eps\to 0} \cJ^\eps(n^\eps) \geq J(\nu)$ for curves of finite action $n^\eps$ converging to $\nu$.
The lower semi-continuity estimate can be established by showing individual semi-continuity estimates for the energy, action and dissipation.
This is the content of Theorem~\ref{thm:Scale:conv:CfA}.
\begin{theorem}[Convergence of curves of finite action]\label{thm:Scale:conv:CfA}
  Suppose that $\alpha\geq 1-3\gamma$.
  For $T>0$ let $(n^\eps,w^\eps)\in \cCE^\eps_T$ be a rescaled curve of finite action and $\nu_0^\eps:= \Pi_{\mac}^\eps n^\eps(0)$ with $\Pi_{\mac}^\eps$ as defined in~\eqref{e:def:BD:mac} satisfy
  \begin{equation}\label{ass:tightness}
    \int_{R}^\infty \lambda \; \nu^\eps_0(\dx\lambda) \to 0 \qquad\text{as } R\to \infty \qquad \text{ uniformly in $\eps$.}
  \end{equation}
  Then, there exists a limiting curve $t \mapsto (\nu_t, w_t)\in \CE_T$ such that
  \begin{equation}\label{e:conv:nu}
    \nu_t^\eps:= \Pi_{\mac}^\eps n^\eps(t)  \stackrel{*}{\rightharpoonup} \nu_t \qquad\text{in } C_c^0(\R_+)^* \qquad \text{for all } t\in [0,T]
  \end{equation}
  and
  \begin{equation}\label{e:conv:mu}
    w^\eps_t(\lambda) \nu_t^\eps(\dx\lambda) \dx{t} \stackrel{*}{\rightharpoonup} w_t(\lambda) \nu_t(\dx\lambda) \dx{t}  \qquad\text{in}\quad C_c^0([0,T]\times \R_+)^* .
  \end{equation}
  There exists $u\in L^2((0,T))$ such that
  \begin{equation}\label{e:Scale:conv:monomers}
    h^\eps(t) := \frac{n_1(t) - z_s}{\eps^{\gamma}} \rightharpoonup u(t) , \qquad\text{ weakly in } L^2((0,T)) ,
  \end{equation}
  and $u(t)$ satisfies the identity
  \begin{equation*}
    u(t) = \frac{q \int \lambda^{\alpha- \gamma} \,\nu_t(\dx{\lambda})}{\int \lambda^\alpha \,\nu_t(\dx{\lambda})} .
  \end{equation*}
  Moreover, the energy, the action and the dissipation satisfy the following $\liminf$ estimates
  \begin{align}
    \forall t\in [0,T] : \qquad \lim_{\eps\to 0} \cF^\eps(\nu^\eps_t)  &\geq   E(\nu_t)  , \label{e:Scale:conv:Fliminf}\\
    \liminf_{\eps\to 0} \int_0^T \cA^\eps(\nu^\eps_t,w^\eps_t) \dx{t}  &\geq   \int_0^T A(\nu_t, w_t) \dx{t} , \label{e:Scale:conv:Aliminf} \\
     \liminf_{\eps\to 0} \int_0^T \cD^\eps(\nu^\eps_t) \dx{t} &\geq  \int_0^T D(\nu_t) \dx{t} .\label{e:Scale:conv:Dliminf}
  \end{align}
\end{theorem}
The classical conclusion from the above theorem is the convergence of curves of maximal slope under the assumption of well-prepared initial data to deal with the term $-\cF^\eps(n^\eps(0))$ inside of $\cJ^\eps(n^\eps)$.
The following Corollary is an immediate consequence of Theorem~\ref{thm:Scale:conv:CfA} by the arguments of~\cite[Theorem 2]{Serfaty2011}.
\begin{corollary}[Convergence of curves of maximal slope]\label{cor:ConvDeGiorgi}
  Suppose $\alpha\geq 1-3\gamma$ and let $(n^\eps,w^\eps)\in \cCE^\eps_T$ be a curve of finite action.
  Moreover assume $\nu_0^\eps:= \Pi_{\mac}^\eps n^\eps(0)$ satisfy the tightness condition~\eqref{ass:tightness} and $n^\eps(0)$ is well-prepared in the sense that
  \begin{equation*}
    \lim_{\eps \to 0} \cF^\eps(n^\eps(0)) = E(\nu_0) .
  \end{equation*}
  Then, there exists a limiting $(\nu, w)\in \CE_T$ satisfying ~\eqref{e:conv:nu} and~\eqref{e:conv:mu} such that
  \begin{equation*}
    \liminf_{\eps\to 0} \cJ^\eps(n^\eps) \geq J(\nu) \geq 0.
  \end{equation*}
  Especially, if $\cJ^\eps(n^\eps)=0$ then $J(\nu) = 0$ and it holds
  \begin{align*}
    \lim_{\eps\to 0} \cF^\eps(n^\eps(t)) &=  E(\nu_t) &&\text{for all } t\in [0,T],\\
    \cA^\eps(n^\eps,w^\eps) &\to A(\nu, w)  &&\text{for a.e.\ } t\in [0,T], \\
    \cD^\eps(n^\eps) &\to  D(\nu) &&\text{for a.e.\ } t\in [0,T]  .
  \end{align*}
\end{corollary}
\subsection{Quasistationary evolution}
The statement~\eqref{e:Scale:conv:monomers} connects the microscopic monomer concentration with a ratio of moments of the macroscopic cluster distribution.
It is possible to show this identity already on the level of rescaled Becker--Döring equation alone.
That is, the monomer concentration follows closely a moment ratio of the distribution of the large clusters.
\begin{proposition}\label{prop:stability}
  For any curve $(n^\eps,\phi^\eps)\in \cCE_T^\eps$ such that $\cJ^\eps(n^\eps) < \infty$ uniformly in $\eps$ and $\nu^\eps_0$ satisfying~\eqref{ass:tightness} the rescaled monomer excess concentration $h^\eps$ as defined in~\eqref{e:Scale:conv:monomers} satisfies
  \begin{equation}\label{e:DissBound}
     \int_0^T \bra*{ h^\eps(t) - u^\eps(t) }^2 \; dt \leq C \int_0^T \cD^\eps_{\mac}(n(t)) \; dt ,
  \end{equation}
  where $\cD_{\mac}^\eps$ is defined like $\cD^\eps$ with summation restricted to $\set{l_0,\dots,\infty}$ and
  \begin{equation*}
     u^\eps(t) := \frac{\sum_{l\geq l_0} \bra*{b_{l+1} n_{l+1}(t) - a_l n_l}}{\eps^{\gamma} \sum_{l\geq l_0} a_l n_l} .
  \end{equation*}
\end{proposition}
The above results together with a refined energy-dissipation estimate based on a logarithmic Sobolev inequality allows to establish detailed information on the distribution of the small clusters for curves of rescaled finite action and in particular for every solution of the time-rescaled Becker--Döring equation~\eqref{e:BD:GF:rescaled}.
The result makes part of the formal asymptotic contained in~\cite[Section 3]{Niethammer2003} rigorous.
\begin{theorem}[Quasistationary distribution]\label{thm:QuasiSmall}
  For any curve $(n^\eps,\phi^\eps)\in \cCE_T^\eps$ such that $\cJ^\eps(n^\eps) < \infty$ uniformly in $\eps$ and $\nu^\eps_0$ satisfying~\eqref{ass:tightness} the small cluster follow a quasistationary distribution dictated by $n_1$: For
  $l_0 = \lfloor\eps^{-x}\rfloor$ with $x$ satisfying~\eqref{ass:cutoff} holds
  \begin{equation*}
     \int_0^T \cH_{\mic}\bra[\big]{ n^\eps(t) \mid \omega(n_1^\eps(t))}  \dx{t} \leq C \eps^{\gamma+(1-x)(1-\alpha+\gamma)} \int_0^T \cD^\eps_{\mic}(n^\eps_t) \dx{t}  \label{e:monexp:q3},
  \end{equation*}
  where $\omega_l(z) = z^l Q_l$ as defined in~\eqref{e:BD:equilibrium}, $\cD_{\mic}^\eps$ is defined like $\cD^\eps$ with summation restricted to $\set{1,\dots, l_0-1}$ and $\cH_{\mic}$ is the microscopic relative entropy defined by
  \begin{equation*}
    \cH_{\mic}(n \mid \omega(z) ) := \sum_{l=1}^{l_0-1} \omega_l(z) \psi\bra*{\frac{n_l}{\omega_l(z)}} \quad\text{with}\quad \psi(x) = x \log x - x +1.
  \end{equation*}
  In particular, for a.e.\ $t\in(0,T)$ it holds
  \begin{equation}\label{e:monexp:micF}
     \lim_{\eps\to 0} \cF_{\mic}^\eps(n^\eps(t)) = 0 \qquad\text{and}\qquad  \lim_{\eps\to 0}  \cF^\eps_{\mac}(\nu^\eps_t) = E(\nu_t) ,
  \end{equation}
  where $\cF_{\mic}^\eps$ is defined like $\cF^\eps$ with summation restricted to $\set{1,\dots,l_0-1}$.
\end{theorem}
\begin{remark}
  The statement~\eqref{e:monexp:micF} is not enough to ensure well-prepared initial data, since the statement only holds for a.e.\ $t\in [0,T]$.
  However, it suggests that the statement of Corollary~\ref{cor:ConvDeGiorgi} holds already under the assumption of \emph{macroscopically well-prepared} initial data:
  \begin{equation}\label{ass:wellprepared:initial:mac}
    \lim_{\eps \to 0} E(\nu^\eps_0) = E(\nu_0) .
  \end{equation}
  The assumption~\eqref{ass:wellprepared:initial:mac} together with the tightness condition~\eqref{ass:tightness} are natural, since they are also a sufficient condition for establishing continuous dependency on the initial data for the limiting gradient flow (cf.\ Corollary~\ref{cor:LSW:continitial}).
\end{remark}
\begin{remark}
  It is possible to use a different rescaling of the Becker--Döring system with different assumptions on the coagulation and fragmentation rates to obtain the LSW equation in the limit (cf.\ \cite{Collet2002,Laurencot2002}). Recently, within this scaling regime a quasi steady approximation was used to derive a suitable boundary condition for the macroscopic limits (cf.\ \cite{Deschamps2016}).
\end{remark}

\section{The LSW equation and its gradient structure}\label{S:LSW}

To make the formal calculation from Section~\ref{s:Intro:MacLim} rigorous, we introduce the concept of curves of finite action for the LSW equation.
\begin{definition}[Curves of finite action]
 A weakly$^*$ continuous curve $[0,T]\ni t \mapsto \nu_t \in M$ is called a curve of finite action, if
 there exists a measurable vector field $[0,T] \ni t \mapsto w_t \in T_{\nu_t}^* M$ such that
 \begin{equation*}
   A(\nu,w) := \int_0^T \int \lambda^{\alpha} |w_t|^2 \; \nu_t(\dx\lambda) < \infty ,
 \end{equation*}
 where the pair $(\nu,w)\in \CE_T$ solves the continuity equation
 \begin{equation}\label{e:LSW:continuityEqu}
    \partial_t \nu_t + \partial_\lambda\bra*{\lambda^\alpha w_t \nu_t} = 0   \quad\text{ in }\quad C_c^\infty([0,T]\times \R_+)^*.
 \end{equation}
\end{definition}
Before formulating the compactness statement, we want to revise the definition of the dissipation~\eqref{e:LSW:dissipation} and generalize it to curves of finite action.
The dissipation acts as a weak upper gradient.
Hence, for a curve of finite action $[0,T]\ni t \mapsto \nu_t \in M$ and using the fact that $w_t \in T^*_{\nu_t} M$ for all $t\in [0,T]$ it formally follows
\begin{equation}\label{e:LSW:DD}\begin{split}
  \abs*{E(\nu_T) - E(\nu_0)} &=  \abs*{ -\int_0^T \int q \lambda^{\alpha-\gamma} w_t \dx{\nu_t} \dx{t} } \leq \int_0^T \lambda^\alpha \abs*{ \int \bra*{u(t) - q \lambda^{-\gamma}} \ w_t \dx{\nu_t}} \dx{t} \\
  &\leq \int_0^T \bra*{ \int  \lambda^\alpha  \bra*{u(\nu_t)-q\lambda^{-\gamma}}^2\dx{\nu_t} }^{\frac{1}{2}} \bra*{ A(\nu_t, w_t)}^\frac{1}{2} \dx{t} ,
\end{split}\end{equation}
where $u(\nu_t)$ is an arbitrary function on $M$.
The choice of $u(\nu_t)$ is fixed by a minimization in $L^2$.
That is, we define the dissipation as the weighted $L^2$-minimal upper gradient for the energy.
Before doing so, we need as an auxiliary result, that a finite dissipation implies the existence of the $\alpha$-moment for a curve of finite action.
\begin{lemma}[Moment estimate]
  Assume $\alpha \geq 1-3\gamma$.
  Let $(\nu,w)\in \CE_T$ be a curve of finite action in $M$ such that
 \begin{equation}\label{e:LSW:minDissipation}
  \inf_{u\in L^2([0,T])} \int_0^T \int \lambda^\alpha  \bra*{u(t) - q \lambda^{-\gamma}}^2 \dx{\nu_t} \dx{t} < \infty .
 \end{equation}
 Then, it holds the moment estimate
 \begin{equation}\label{e:LSW:MomentDissipation}
   \int_0^T \int \lambda^{\alpha} \dx{\nu_t} \dx{t} < \infty .
 \end{equation}
\end{lemma}
\begin{proof}
  Let us define $D(\nu,u) = \int \lambda^\alpha \bra*{u - q \lambda^{-\gamma}}^2 \dx{\nu}$.
  We observe that for $\alpha \geq 1-\gamma$, there is nothing to show, since the bound follows by interpolation from $\sup_{t\in [0,T]} E(\nu_t) < \infty$ and $\int \lambda \dx{\nu_t} =\bar\varrho$.

  Therefore, assume now $\alpha \leq 1-\gamma$.
  Let us define $\eta(\lambda) := \lambda \chi_{[0,1]}(\lambda) + \chi_{(1,\infty)}(\lambda)$.
  Then, we can estimate with Cauchy--Schwarz for any $\kappa \in \R$
  \begin{equation}\label{e:LSW:DissEst:ap}
  \begin{split}
    &\int_0^T \bra*{ \int \bra*{u(t) - q\lambda^{-\gamma}} \eta(\lambda)^\kappa \dx{\nu_t}}^2 \dx{t}\leq \int_0^T D(\nu_t,u(t)) \int \eta(\lambda)^{2\kappa} \lambda^{-\alpha} \dx{\nu_t} \dx{t} \\
    &\qquad\qquad \leq  \int_0^T D(\nu_t, u(t)) \dx{t}  \ \sup_{t\in [0,T]} \int \eta(\lambda)^{2\kappa} \lambda^{-\alpha}\dx{\nu_t} .
  \end{split}
  \end{equation}
  Since, $\sup_{t\in [0,T]} E(\nu_t) < \infty$ and $\int \lambda \dx{\nu_t} =\bar\varrho$, we can use interpolation to bound the $\sup$ in $t$ provided $2\kappa -\alpha \geq 1-\gamma$.
  On, the other hand, since $\int \lambda \dx{\nu_t} = \bar\varrho$ for all $t\geq 0$, there exists a constant $\bar\varrho_T > 0$ for any $T>0$ such that $\int \eta(\lambda) \dx{\nu_t} \geq \bar\varrho_T$ (see also Lemma~\ref{lem:LowerMoment} for a similar argument).
  We can estimate the left hand side of~\eqref{e:LSW:DissEst:ap} from below in the case $\kappa=1$ by using the Young inequality for some $0<\tau <1$
  \[\begin{split}
    \int_0^T \bra*{ \int \bra*{u(t) - q\lambda^{-\gamma}} \eta(\lambda) \dx{\nu_t}}^2 \dx{t} &\geq \bra*{1-\tau}  \bar\varrho_T^2 \int_0^T u(t)^2 \dx{t} \\
    &\phantom{\geq} - \bra*{\frac{1}{\tau}-1} \int_0^T \bra*{(1-\gamma) E(\nu_t)}^2 \dx{t} .
  \end{split}\]
  Since, $E(\nu_t)\in L^\infty([0,T])$, we obtain the first a priori estimate
  \begin{equation}\label{e:LSW:DissEst:ap1}
    \int_0^T u(t)^2 \dx{t} \leq C_T \int_0^T D(\nu_t, u(t))\dx{t} + C_T .
  \end{equation}
  Another choice is $\kappa=1-\gamma$ thanks to $\alpha \leq 1-\gamma$.
  Then, we estimate the left hand side of~\eqref{e:LSW:DissEst:ap} by using again the Young inequality with $\tau\in (0,1)$ as follows
  \[\begin{split}
    \int_0^T \bra*{ \int \bra*{u(t) - q\lambda^{-\gamma}} \eta(\lambda)^{1-\gamma} \dx{\nu_t}}^2 \dx{t} &\geq \bra*{1-\tau} q \int \bra*{\int_0^T \lambda^{1-2\gamma} \dx{\nu_t}}^2 \dx{t} \\
    &\phantom{\geq} - \bra*{\frac{1}{\tau}-1} \int_0^T u(t)^2 \bra*{(1-\gamma) E(\nu_t)}^2 \dx{t} .
  \end{split}\]
  Since, we trivially have $\int_1^\infty \lambda^{1-2\gamma}\dx{\nu_t} \leq \int \lambda \dx{\nu_t} = \bar\varrho$, it follows by using the first a priori bound~\eqref{e:LSW:DissEst:ap1} and $E(\nu_t) \in L^\infty([0,T])$ the second a priori estimate
  \begin{equation}\label{e:LSW:DissEst:ap2}
    \int_0^T \bra*{ \int \lambda^{1-2\gamma} \dx{\nu_t}}^2 \dx{t} \leq C_T \int_0^T D(\nu_t, u(t))\dx{t} + C_T ,
  \end{equation}
  which shows~\eqref{e:LSW:MomentDissipation} for $\alpha \geq 1-2\gamma$.
  Hence, we assume now $\alpha \leq 1-2\gamma$.
  Similarly to~\eqref{e:LSW:DissEst:ap}, we can now estimate by Cauchy--Schwarz for some $\tilde \kappa \in \R$
  \begin{equation}\label{e:LSW:DissEst:ap3}
  \begin{split}
    \int_0^T \abs*{ \int \bra*{u(t) - q\lambda^{-\gamma}} \eta(\lambda)^{\tilde\kappa} \dx{\nu_t}} \dx{t} &\leq \bra*{\int_0^T D(\nu_t,u(t))\dx{t}}^\frac12 \times \\
     &\qquad\bra*{\int_0^T \int \eta(\lambda)^{2\tilde\kappa} \lambda^{-\alpha} \dx{\nu_t} \dx{t} }^\frac12 .
  \end{split}
  \end{equation}
  The second factor is bounded for $2\tilde \kappa  - \alpha \geq 1-2\gamma$ by~\eqref{e:LSW:DissEst:ap2}.
  Hence, a possible choice is $\tilde \kappa = 1-2\gamma$ by the assumption $\alpha \leq 1-2\gamma$.
  Since $u \in L^2((0,T))$ and $\int \lambda^{1-2\gamma} \dx{\nu_t} \in L^2((0,T))$, we conclude the estimate~\eqref{e:LSW:MomentDissipation}.
\end{proof}
The Lemma provides the crucial ingredient to conclude that the dissipation is well-defined and justifies the use of the weak formulation in the first step of~\eqref{e:LSW:DD}.
\begin{proposition}\label{lem:LSW:Dissipation}
  Assume $\alpha \geq 1-3\gamma$.
  Let $(\nu,w)\in \CE_T$ be a curve of finite action in~$M$ such that~\eqref{e:LSW:minDissipation} holds.
  Then the associated minimization problem has a unique solution $u\in L^2([0,T])$ such that
  \begin{equation}\label{e:LSW:covector}
    \lambda \mapsto u(t) - q\lambda^{-\gamma} \in T^*_{\nu_t} M \qquad \text{ for a.e.\ } t\in [0,T].
  \end{equation}
  Moreover, the associated functional defined for a.e.\ $t\in [0,T]$ by
  \begin{equation}\label{e:LSW:Dissipation}
    D(\nu_t) := \int \lambda^\alpha \bra*{u(t)-q \lambda^{-\gamma}}^2 \dx{\nu_t} \quad\text{with}\quad u(t) :=  \frac{q\int\lambda^{\alpha-\gamma} \dx{\nu_t}}{\int \lambda^\alpha \dx{\nu_t}},
  \end{equation}
  called \emph{dissipation}, is a~\emph{strong upper gradient} for the energy $E$.
  That is, it holds for any curve $(\nu,w)\in \CE_T$ of finite action
 \begin{equation}\label{e:LSW:strongupper}
  \abs{ E(\nu_t) - E(\nu_s) } \leq \int_s^t \sqrt{D(\nu_r)} \, \sqrt{A(\nu_r,w_r)} \, \dx{r} , \qquad \forall 0\leq s< t \leq T.
 \end{equation}
 Hereby, equality in~\eqref{e:LSW:strongupper} holds if and only if $w_t(\lambda) = \pm\bra{u(t) - q\lambda^{-\gamma}}$ for $\nu_t$-a.e.\ $\lambda \in \R_+$.
\end{proposition}
\begin{proof}
  In the first step, we show~\eqref{e:LSW:covector} and~\eqref{e:LSW:Dissipation}.
  Therefore, the first variation of the minimization problem~\eqref{e:LSW:minDissipation} along some $s:\R_+ \to \R$ is given by
  \begin{equation*}
    \int_0^T \int  \bra*{u(t) - q \lambda^{-\gamma}} \lambda^\alpha \dx{\nu_t} \, s(t) \dx{t} = 0 .
  \end{equation*}
  We show that is is well-defined by an estimate analog to~\eqref{e:LSW:DissEst:ap3}
  \[\begin{split}
    \abs*{\int_0^T \lambda^\alpha  \bra*{u(t) - q \lambda^{-\gamma}} \dx{\nu_t}\; s(t) \dx{t}} &\leq \bra*{ \int_0^T D(\nu_t, u(t)) \dx{t}}^\frac12 \times \\
    &\qquad\bra*{\int_0^T s(t)^2 \int \lambda^{\alpha} \dx{\nu_t}\dx{t}}^\frac12 ,
  \end{split}\]
  which is bounded thanks to the estimate~\eqref{e:LSW:MomentDissipation} for $s\in L^\infty((0,T))$.
  In addition the a prior estimate~\eqref{e:LSW:DissEst:ap1} shows that minimizer is actually in $L^2((0,T))$ and hence satisfying the Euler-Lagrange equation $\int \bra*{u(t)-q \lambda^{-\gamma}} \lambda^\alpha \dx{\nu_t} = 0$ for a.e.\ $t \in [0,T]$, which is nothing else than~\eqref{e:LSW:covector} also showing~\eqref{e:LSW:Dissipation}.

  It is left to show, that $D(\nu_t)$ is a strong upper gradient for the energy.
  Therefore, we fix a test function $\zeta\in C_c^\infty(\R_+)$ and calculate for a curve $(\nu,w)\in \CE_T$
  \begin{align*}
    \pderiv{}{t} \frac{q}{1-\gamma} \int \lambda^{1-\gamma} \zeta \, \dx\nu_t &= q \int \lambda^{\alpha-\gamma} \zeta \, w_t \, \dx\nu_t + \frac{q}{1-\gamma} \int \lambda^{1+\alpha-\gamma} \zeta' \, w_t \, \dx\nu_t =: \I + \II.
  \end{align*}
  Using the fact that $w_t\in T_{\nu_t}^* M$, we can smuggle in $u(t)$ and apply Cauchy--Schwarz to the first term $\I$, to obtain
  \begin{align*}
    \I \leq { A(\nu_t, w_t)}^\frac12 \ \bra*{ \int \lambda^{\alpha} \bra{ u - q\lambda^{-\gamma} \zeta}^2 \dx{\nu_t}}^\frac12,
  \end{align*}
  Hereby, equality holds if and only if $w_t = \pm w^\zeta_t$ with $w^\zeta_t := u - q\lambda^{-\gamma} \zeta$.
  Hence, by choosing $\zeta_n$ converging to $1$ from below the result~\eqref{e:LSW:strongupper} follows by integration in time and dominated convergence, provided the term $\II$ vanishes.
  By an additional approximation step, we can justify to choose the sequence $\zeta_n(\lambda) = n \lambda \chi_{[0,1/n)} + \chi_{[1/n,\infty)}$ and estimate~$\II$ by
  \begin{align*}
    \II \leq  \frac{1}{1-\gamma} \bra*{ \int_0^{\frac{1}{n}} \lambda^{\alpha} \abs*{w_t}^2 \dx{\nu_t}}^\frac12\ \bra*{ \int_0^{\frac{1}{n}} \lambda^{\alpha} \abs*{ u - q\lambda^{-\gamma} \lambda\zeta'_n}^2   \dx{\nu_t}}^\frac12.
  \end{align*}
  Since, we can assume the r.h.s.\ of~\eqref{e:LSW:strongupper} to be finite, we can conclude again by dominated convergence, that $\II \to 0$ as $n\to \infty$, which finishes the proof.
\end{proof}
\begin{lemma}[Tightness is preserved by curves of finite action]\label{lem:LSW:tight}
  Let $\set{\nu_0^\eps}_{\eps>0}$ be a family satisfying the tightness condition~\eqref{ass:tightness}.
  Then for any $T>0$ and any family of curves $\set{(\nu^\eps,w^\eps)\in \CE_T : \nu^\eps_{t=0} = \nu^\eps_0 }_{\eps>0}$ of uniformly finite action the family $\set{\nu_t^\eps}_{\eps>0}$ satisfies the tightness condition~\eqref{ass:tightness} uniformly in $\eps$ for any $t\in [0,T]$.
\end{lemma}
\begin{proof}
  Fix a test function $\eta_{r,R} \in C_c^\infty(\R_+,[0,1])$ such that $\eta_{r,R}(s) = 0$ for $s<r/2$ and $s>2R$, $\eta_{r,R}(s) = 1$ for $r\leq s \leq R$, $\abs{\eta'_{r,R}(s)} \leq C/r$ for $r/2\leq s < r$ as well as $\abs{\eta'_{r,R}(s)} \leq C /R$ for $R\leq s < 2R$.
  We can estimate for a fixed curve of finite action $(\nu,w)\in \CE_T$
  \begin{align*}
    \abs*{\pderiv{}{t} \int \lambda \eta(\lambda) \dx{\nu_t} } &\leq \int \lambda^\alpha \, \eta \, \abs{w_t} \dx{\nu_t} + \int \lambda^{1+\alpha} \, \abs{\eta'} \, \abs{w_t} \dx{\nu_t} \\
    &\leq \frac{C}{r^{\frac{1-\alpha}{2}}} \int_{\frac{r}{2}}^\infty \lambda^{\frac{1+\alpha}{2}} \abs{w_t} \dx{\nu_t} + \frac{C}{r^{\frac{1-\alpha}{2}}} \int_{\frac{r}{2}}^r \lambda^{\frac{1+\alpha}{2}} \abs{w_t} \dx{\nu_t} \\
    &\quad +\frac{C}{R^{\frac{1-\alpha}{2}}} \int_{R}^{2R} \lambda^{\frac{1+\alpha}{2}} \abs{w_t} \dx{\nu_t} \\
    &\leq C\bra*{\frac{1}{r^{\frac{1-\alpha}{2}}}+ \frac{1}{R^{\frac{1-\alpha}{2}}}}\  A(\nu_t,w_t)^\frac12 \ \bra*{ \int \lambda \dx{\nu_t} }^{\frac{1}{2}}.
  \end{align*}
  By an integration in time, letting $R\to \infty$ and using the assumption of finite action, we obtain for all $t\in [0,T]$ the estimate
  \begin{equation*}
     \int_{r}^\infty \lambda \dx{\nu_t} \leq \int_{\frac{r}{2}}^\infty \lambda \dx{\nu_0} + \frac{C \sqrt{T}}{r^\frac{1-\alpha}{2}} .
  \end{equation*}
  Hereby the constant $C$ only depends on the test function and the action of the curve. Hence, if we apply this estimate for $\set{\nu_t^\eps}_{\eps>0}$, we observe its tightness by the tightness assumption on $\set{\nu_0^\eps}_{\eps>0}$ and the uniform finite action of the family.
\end{proof}
\begin{proposition}[Compactness of curves of finite action]\label{prop:LSW:compactness}
  Assume $\alpha \geq 1-3\gamma$ and let $(\nu^n, w^n) \in \CE_T$ for $n\in\N$ be a family of solutions to the continuity equation with uniformly bounded action and dissipation such that $\set{\nu^n_0}_{n\in\N}$ satisfies the tightness condition~\eqref{ass:tightness}.
  Then, there exists a subsequence and a couple $(\nu, w)\in \CE_T$, such that
\begin{align}
  \nu_t^n &\stackrel{*}{\rightharpoonup} \nu_t \qquad \text{ in } C_c^0(\R_+)^* \quad \forall t\in [0,T] ,  \label{e:LSW:compact:nu} \\
  w^n \nu^n    &\stackrel{*}{\rightharpoonup}  w \nu \qquad \text{ in } C_c^0([0,T] \times \R_+)^*.\notag
\end{align}
  In addition, the action and dissipation satisfy the $\liminf$ estimates
\begin{align}\label{e:LSW:compact:Alsc}
   \int_0^T A(\nu_t, w_t) \dx{t} &\leq \liminf_{n\to \infty} \int_0^T A(\nu_t^n, w_t^n)  \dx{t} \\
  \int_0^T D(\nu_t) \dx{t} &\leq \liminf_{n\to \infty} \int_0^T D(\nu_t^n) \dx{t} \label{e:LSW:compact:Dlsc}.
\end{align}
\end{proposition}
\begin{proof}
  For any $\zeta\in C_c^1(\R_+)$ and $0\leq t_1 < t_2 \leq T$ holds
\begin{align*}
  \abs*{\int \zeta \dx{\nu_{t_2}^n} - \int \zeta \dx{\nu_{t_1}^n}} &= \abs*{ \int_{t_1}^{t_2} \int \partial_\lambda \zeta \, \lambda^\alpha \, w_t^n \; \nu_t(\dx{\lambda}) \dx{t}}  \\
  &\leq \sup_{\lambda > 0} \frac{\abs*{\partial_\lambda \zeta}}{\lambda^{\frac{\alpha}{2}}} \abs*{t_2 - t_1}^\frac12 \bra*{\int_{t_1}^{t_2} A(\nu_t^n, w_t^n) \dx{t}}^\frac12 ,
\end{align*}
  which shows~\eqref{e:LSW:compact:nu} and the weak$^*$ continuity of $\nu_t$.
  Moreover, it holds for $\kappa \in \R$ and $\zeta \in C_c^0([0,T]\times \R_+)$
\begin{align}
  \int_0^T \int \zeta(t,\lambda) \lambda^{\kappa+\alpha} \abs*{w_t^n} \dx\nu_t^n \dx{t} \leq &\bra*{ \int_0^T A(\nu_t^n,w_t^n) \dx{t}}^\frac12 \times \notag \\
   &\bra*{ \int_0^T \int \zeta(t,\lambda)^2 \lambda^{2\kappa+\alpha} \dx\nu_t \dx{t} }^\frac12   \label{e:LSW:compact:p1} .
\end{align}
  By lower semi-continuity it follows $E(\nu_t)<\infty$ and by the tightness Lemma~\ref{lem:LSW:tight} it follows the conservation of total mass $\int \lambda \dx{\nu_t} = \int \lambda \dx{\nu_0} = \int \lambda \dx{\nu^n_0} < \infty$.
  Hence, $\nu_t \in M$ for all $t\in [0,T]$.
  Then, by interpolation, the second term in~\eqref{e:LSW:compact:p1} is finite for $1-\gamma\leq 2\kappa +\alpha \leq 1$.
  There exists $\mu \in C_c^0([0,T]\times \R_+)$ such that $w^n \nu^n \stackrel{*}{\rightharpoonup} \mu$ and the pair $(\nu,\mu)$ satisfies $\partial_t \nu_t + \partial_{\lambda} \mu_t = 0$ in $C_c^\infty([0,T]\times \R_+)^*$.
  Since $(\nu^n,w^n)$ is a curve of finite action, we find a subsequence such that
  \[
    \lim_{n\to \infty} \int_0^T A(\nu^n_t,w^n_t) \dx{t} = A^* := \liminf_{n\to \infty} \int_0^T  \int_0^T A(\nu^n_t,w^n_t) \dx{t} .
  \]
  Hence, we get the estimate with $\kappa = (1-\alpha)/2$
  \begin{equation}\label{e:LSW:compact:Action:p1}
    \int_0^T \int  \zeta(t,\lambda) \lambda^\frac{1-\alpha}{2} \mu_t(\dx\lambda) \dx{t} \leq \bra*{A^* \int_0^T \int \zeta(t,\lambda)^2 \lambda \dx{\nu_t} \dx{t}}^\frac12.
  \end{equation}
  Now, we can apply the Riesz representation theorem to find $v\in L^2(\lambda\dx{\nu_t} dt)$ such that
  \[
    \int_0^T \int \zeta(t,\lambda)  \lambda^\frac{1-\alpha}{2} \mu_t(\dx\lambda) \dx{t} = \int_0^T \int \zeta(t,\lambda) \lambda v(t,\lambda) \; \nu_t(\dx{\lambda}) \dx{t} .
  \]
  Setting $\tilde\zeta(t,\lambda) = \lambda^\frac{1-\alpha}{2} \zeta(t,\lambda)$ and $w_t(\lambda) =\lambda^{\frac{1-\alpha}{2}} v(t,\lambda)$, we get that $\mu_t(\dx\lambda) = v(t,\lambda) \nu_t(\dx\lambda)$.
  Moreover, since $w\in L^2(\lambda^\alpha \dx\nu_t \dx{t})$ it is of finite action.
  Moreover, by approximating $\zeta(t,\lambda) = \frac{w_t(\lambda)}{\lambda^{\frac{1-\alpha}{2}}}$ it follows from~\eqref{e:LSW:compact:Action:p1} the lower semi-continuity of the action~\eqref{e:LSW:compact:Alsc}.

  Finally, \eqref{e:LSW:compact:Dlsc} follows by noting that $D(\nu_t^n) = A(\nu_t^n, u(\nu_t^n) - q\lambda^{-\gamma})$, which is well-defined by~\eqref{e:LSW:covector}.
\end{proof}
The formulation of the LSW gradient flow as curves of minimal action, reads now in analog to the one of the Becker-Döring equation~\eqref{e:BD:Jintro}
\begin{proposition}[LSW equation as curves of maximal slope]\label{prop:LSW:DeGiorgi}
  Let $\alpha \geq 1-3\gamma$.
  For $(\nu,w) \in \CE_T$ with finite action holds
  \begin{align}\label{e:LSW:J}
    J(\nu) &:= E(\nu_T) - E(\nu_0) + \frac{1}{2} \int_0^T
     D(\nu_t) \dx{t}
    + \frac{1}{2} \int_0^T A(\nu_t, w_t) \dx{t} \geq 0 .
  \end{align}
  Moreover, equality holds if and only if $\nu_t$ is a solution to the LSW equation.
\end{proposition}
\begin{proof}
  We can assume that the dissipation $\int_0^T D(\nu_t) \dx{t}$ is bounded, because else there is nothing to show.
  Then, we can use the strong upper gradient property of the dissipation~\eqref{e:LSW:strongupper} after an application of the Young inequality to arrive at
\begin{equation*}
  \pderiv{}{t} E(\nu_t) \geq -\frac{1}{2} D(\nu_t)  - \frac{1}{2} A(\nu_t,w_t) .
\end{equation*}
  An integration of the above estimate shows the nonnegativity of $J$ in~\eqref{e:LSW:J}.
  The equality case follows from the equality case in~\eqref{e:LSW:strongupper} for a.e.\ $t\in [0,T]$ by choosing $w_t(\lambda) = u(\nu_t) - ql^{-\gamma}$.
  Then, by weak$^*$ continuity of $t\mapsto \nu_t$ follows the result for all $t\in [0,T]$.

  Now, for a curve $(\nu,w)\in \CE_T$ with $J(\nu)=0$ follows by Proposition~\eqref{lem:LSW:Dissipation} and~\eqref{e:LSW:strongupper} the identity
  \begin{align*}
    -\int \lambda^{\alpha} q \lambda^{-\gamma} w_t \dx{\nu_t} = \sqrt{A(\nu_t,w_t) D(\nu_t)} = A(\nu_t,w_t) = D(\nu,w_t) .
  \end{align*}
  Since $w_t \in T^*_{\nu_t} M$, it follows that $w_t = u(t)-q\lambda^{-\gamma}$ for $\nu_t$ almost every $\lambda \in \R_+$.
  Hence, the continuity equation~\eqref{e:LSW:continuityEqu} takes the form
  \begin{equation*}
    \partial_t \nu_t + \partial_\lambda\bra*{\lambda^\alpha\bra{ u(t)-q\lambda^{-\gamma}}\nu_t} = 0  \quad\text{ in }\quad C_c^\infty([0,T]\times \R_+)^*,
  \end{equation*}
  which is nothing else than a weak solution to the LSW equation.
\end{proof}
\begin{remark}
  The compactness statement in Proposition~\ref{prop:LSW:compactness} is also a tool to proof existence of solution to the LSW equation by the particle method (cf.\ \cite{Niethammer1998,Niethammer2005b}).
  Therefore, the initial distribution is approximated in the weak$^*$ sense by a discrete sum of Dirac deltas. Solutions for such data are determined by solving the finite system of ordinary differential equations determined by~\eqref{e:LSW} for each particle. Then the compactness statement allows to pass to the limit in the particle number and existence for measure valued initial distributions is obtained.
\end{remark}
In addition, the compactness statement Proposition~\ref{prop:LSW:compactness} with the variational characterization of solutions of the LSW equation from Proposition~\ref{prop:LSW:DeGiorgi} is the essential tool to show the continuous dependence of the solution on the initial data.
\begin{corollary}[Continuous dependency on the initial data]\label{cor:LSW:continitial}
  Let $\set{\nu_0^\eps}_{\eps>0}$ be a sequence of initial data satisfying the tightness condition~\eqref{ass:tightness} and
  \begin{equation}\label{e:LSW:conv:init}
    \lim_{\eps \to 0} E(\nu_0^\eps) = E(\nu_0) < \infty .
  \end{equation}
  Then there exists a solution $\nu \in C_c^\infty([0,T]\times \R_+)^*$ to the LSW equation such that $\nu_t^\eps \stackrel{*}{\rightharpoonup} \nu_t$ in $C_c^0(\R_+)$ for all $t\in [0,T]$.
\end{corollary}
\begin{proof}
  By the compactness statement Proposition~\eqref{prop:LSW:compactness} follows that there exists a couple $(\nu_t, w_t)\in\CE_T$ being the weak$^*$ limit of $(\nu_t^\eps,w_t^\eps)\in\CE_T$ and satisfying the two $\liminf$ estimates~\eqref{e:LSW:compact:Alsc} and~\eqref{e:LSW:compact:Dlsc}.
  By lower semi-continuity of the energy and the assumption~\eqref{e:LSW:conv:init} follows $0=\liminf_{\eps\to 0} J(\nu^\eps) \geq J(\nu) \geq 0$ and hence $J(\nu) =0$, which proves the claim.
\end{proof}
\begin{remark}
  The above result is consistent with the existing literature: In~\cite[Theorem 2.2]{Niethammer2005b}, the continuous dependency on the initial data was shown under the tightness condition~\eqref{ass:tightness} with respect to weak$^*$ convergence for continuous test functions compactly supported on $[0,\infty)$ including~$0$, i.e.\ Borel measures on $[0,\infty)$.
  Then, it is easy to see that weak$^*$ convergence with respect to this class implies convergence of the macroscopic energy~\eqref{ass:wellprepared:initial:mac}.
\end{remark}

\section{Proof of main results}\label{S:Limit}

\subsection{A priori estimates for the Becker--Döring gradient structure}\label{s:AprioriBD}

In this section, we consider the Becker--Döring equation and its gradient structure as introduced in Section~\ref{s:Intro:BD} and~\ref{s:Intro:BD:GF}, respectively.

The reversible equilibrium distribution $\omega$ with parameter $z \in (0,z_s]$ (corresponding to the conserved quantity) is given by~\eqref{e:BD:equilibrium}. Note, that the radius of convergence for $z \mapsto \sum_{l=1}^\infty l \omega_l(z)$ is $z_s$ and $\sum_{l=1}^\infty l \omega_l =: \varrho_s < \infty$ (cf.\ Lemma~\ref{lem:BD:assymptotic_Q} below).
Hence, the equilibrium state $\omega_l := \omega_l(z_s)$ is the one with largest total mass $\varrho_s$. We work in the excess mass regime and any state will have total mass larger than $\varrho_s$ to which there doesn't exist an according equilibrium state with the same total mass.
The free energy is always the relative entropy with respect to $\omega = \omega(z_s)$, if not stated explicitly.
\begin{lemma}\label{lem:BD:assymptotic_Q}
Under Assumption~\ref{ass:BD:rates}, there exists a constant $\cF_0$ such that for any $l\geq 2$
\begin{equation}\label{e:BD:assymptotic_Q}
  Q_l = \frac{1}{l^\alpha z_s^{l-1}} \exp\bra*{\bra*{\cF_0- \frac{q}{z_s}(1-\gamma) l^{1-\gamma} + \frac{q^2}{2 z_s^2(1-2\gamma)} \bra*{l^{1-2\gamma}-1}} \bra*{1+ O(l^{-\gamma})}} ,
\end{equation}
where $\frac{l^{1-2\gamma}-1}{1-2\gamma} := \log l$ for $\gamma=\frac{1}{2}$.
\end{lemma}
The proof relies on elementary estimates and is included for convenience in Appendix~\ref{s:assymptotic_Q}.
The expansion of the rates allows us to easily conclude the expansion of the free energy~$\cF$.
\begin{lemma}[Expansion of free energy]\label{lem:BD:expansionF}
 Let $n\in \cM$ be given such that $\cF(n)< \infty$ as defined in~\eqref{e:def:Lyapunov}, then there exists $\sigma>0$ such that for any $l_0 \geq 2$
\begin{equation}\label{e:BD:expansionF}
 \cF_{l_0}(n) = \cF_{l_0}^{\LSW}(n) \bra*{1+ O(l_0^{-\sigma}) + O(l_0^\gamma \omega_{l_0})} ,
\end{equation}
where $\cF_{l_0}$ and $\cF_{l_0}^{\LSW}$ are defined by
\begin{equation*}
 \cF_{l_0}(n) := \sum_{l=l_0}^\infty \omega_l \psi\bra*{\frac{n_l}{\omega_l}} \qquad\text{and}\qquad \cF_{l_0}^{\LSW}(n) := \frac{q}{z_s(1-\gamma)} \sum_{l=l_0}^\infty l^{1-\gamma} n_l .
\end{equation*}
\end{lemma}
\begin{proof}
We expand the function $\psi$ in the definition of $\cF_{l_0}$
\begin{align*}
 \cF_{l_0}(n) = \sum_{l=l_0}^\infty \bra*{ n_l \log \frac{1}{z_s^l Q_l} + n_l\bra*{\log n_l - 1 } + \omega_l}
\end{align*}
We estimate the first sum using the asymptotic expansion~\eqref{e:BD:assymptotic_Q}
\begin{align*}
  \sum_{l=l_0}^\infty \bra*{ n_l \log \frac{1}{z_s^l Q_l} } &= \frac{q}{z_s(1-\gamma)} \sum_{l=l_0}^\infty l^{1-\gamma} n_l + O\bra*{\sum_{l=l_0}^\infty l^{1-2\gamma} n_l} \\
  &\quad +\sum_{l=l_0}^\infty l^{1-\gamma} n_l \frac{\log l^\alpha}{l^{1-\gamma}} - \sum_{l=l_0}^\infty \frac{l^{1-\gamma}}{l^{1-\gamma}}n_l \bra*{\log z_s + \cF_0\bra*{1+O(l^{-\gamma})}}\\
  &= \bra*{\frac{q}{z_s(1-\gamma)} \sum_{l=l_0}^\infty l^{1-\gamma} n_l} \bra*{1+ O\bra[\big]{l_0^{-\gamma} \log l_0} +O\bra[\big]{l_0^{-(1-\gamma)}}} ,
\end{align*}
Likewise, we note that for any $\beta \in (0,1)$ exists $C_\beta > 0 $ such that for $x>0$
\begin{equation*}
  \abs*{\min\set*{x (\log x -1),0}} \leq C_\beta x^\beta
\end{equation*}
and with the Hölder inequality, we can estimate
\begin{align*}
 \sum_{l=l_0}^\infty \abs*{\min\set*{ n_l\bra*{\log n_l -1},0}} &\leq C_\beta \sum_{l=l_0}^\infty n_l^\beta \leq C_\beta \bra*{ \sum_{l=l_0}^\infty l^{1-\gamma} n_l}^\beta \bra*{\sum_{l=l_0}^\infty \frac{1}{l^{\frac{\beta}{1-\beta} (1-\gamma)}}}^{1-\beta} .
\end{align*}
Now, we can choose $\beta$ such that $\frac{\beta}{1-\beta} (1-\gamma)= \kappa > 1 $ and $\beta<1$ leading to the estimate
\begin{align*}
 \sum_{l=l_0}^\infty \abs*{\min\set*{ n_l\bra*{\log n_l -1},0}} &\leq C_\beta \;\bra*{ \beta \sum_{l=l_0}^\infty l^{1-\gamma} n_l + 1-\beta}\; O(l_0^{-(\kappa-1)(1-\beta)}).
\end{align*}
The last term evaluates with the help of~\eqref{e:BD:assymptotic_Q} to
\begin{align*}
 \sum_{l=l_0}^\infty \omega_l &\leq \sum_{l=l_0}^\infty \frac{z_s}{l^\alpha} \exp\bra*{\bra*{\cF_0 - \frac{q}{z_s(1-\gamma)} l^{1-\gamma}}\bra*{1+O(l^{-\gamma})}}\\
 &\leq C \int_{l_0}^\infty \exp\bra*{-  \frac{q}{z_s(1-\gamma)} l^{1-\gamma}} \dx{l} \leq C \, l_0^\gamma \exp\bra*{-  \frac{q}{z_s(1-\gamma)} l_0^{1-\gamma}}.
\end{align*}
Therefore, a combination of all the estimates leads to the result.
\end{proof}
Moreover, we need a Czisar-Pinsker inequality for the free energy, which was already a crucial ingredient in~\cite{Niethammer2003}
\begin{proposition}[{Czisar-Pinsker inequality~\cite[Lemma 2.1, 2.2]{Niethammer2003}}]\label{lem:Apriori}
 For $n\in \cM$ and any small $\eta > 0$ and any $p< \infty$ and any $l_0\geq 2$ holds
\begin{align}
 \sum_{l=1}^\infty l^{1-\gamma} \abs*{n_l - \omega_l} &\leq C \sqrt{\cF(n)} \label{e:Apriori1}\\
 \abs*{\sum_{l=l_0}^\infty l n_l - \bra*{\varrho-\varrho_s}} &\leq C l_0^\gamma \sqrt{\cF(n)} + C_{p} l_0^{-p} . \label{e:Apriori3}
\end{align}
\end{proposition}
For the next Lemmata, we make statements on curves of finite action to deduce certain compactness, which we later need for passing to the limit.
These Lemmata are the analog of~\cite[Lemma 2.3 and 2.4]{Niethammer2003}, but we proof them for curves of finite action instead of solutions to the Becker--Döring equation.
\begin{lemma}[A priori estimates for curves of finite action]
  Let $(n,\phi)\in \cCE_T$ be a curve of finite action as in Definition~\ref{def:BD:CurvesFiniteAction} and $\eta\in L^2(0,T)$, then it holds
\begin{align}\label{e:L1ActionEst}
  \int_0^T \eta(t) \sum_{l=l_0}^\infty l^\kappa k^l \hat n^\omega_l(t) \abs*{\nabla_l \phi(t)} \dx{t} \leq C\; &\bra[\bigg]{\sup_{t\in [0,T]} \cF^{\LSW}_{l_0}(n(t))}^{\frac{1-\alpha-2\kappa}{2\gamma}} \\
   &\times \int_0^T \abs*{\eta(t)} \; \sqrt{\cA_{\mac}(n(t),\phi(t))} \dx{t},\notag
\end{align}
for any $\kappa \in \pra*{\frac{1-\alpha-\gamma}{2},\frac{1-\alpha}{2}}$ with $\cA_{\mac}$ the action as defined in \eqref{e:BD:action} restricted to $l\geq l_0$.
Hereby, $\nabla_l \phi := \phi_{l+1} - \phi_l - \phi_1$ and  $\hat n_l^\omega$ is defined in~\eqref{e:def:hatn}. Moreover, it also holds the estimate
\begin{align}\label{e:L1FluxEst}
  \int_0^T \eta(t) \sum_{l=l_0}^\infty l^\kappa \abs*{ a_l n_1(t) n_l(t) - b_{l+1} n_{l+1}(t)} \leq C\; &\bra[\bigg]{\sup_{t\in [0,T]} \cF^{\LSW}_{l_0}(n(t))}^{\frac{1-\alpha-2\kappa}{2\gamma}} \\
   &\times \int_0^T \abs*{\eta(t)} \; \sqrt{\cD_{\mac}(n(t))} \dx{t} ,\notag
\end{align}
where again $\cD_{\mac}$ is defined as in~\eqref{e:BD:Dissipation} restricted to $l\geq l_0$.
\end{lemma}
\begin{proof}
  We estimate using the Cauchy--Schwarz inequality
\begin{equation*}
 \sum_{l=l_0}^\infty l^\kappa k^l \hat n^\omega_l \abs*{\nabla_l\phi(t)} \leq \bra*{\sum_{l=l_0}^\infty k^l \hat n^\omega_l \abs*{\nabla_l\phi(t)}^2}^\frac12 \bra*{\sum_{l=l_0}^\infty l^{2\kappa}  k^l \hat n^\omega_l}^\frac{1}{2} .
\end{equation*}
 Now, using that fact that
 \[
   k^l \hat n^\omega_l = \Lambda\bra*{a_l n_1 n_l, b_{l+1} n_{l+1}} =\Lambda\bra*{  \lambda^{\alpha}  n_1 n_l ,  (\lambda+1)^{\alpha} (z_s+q (l+1)^{-\gamma}) n_{l+1}},
 \]
 the estimate $\frac{(\lambda+1)^\alpha}{\lambda^{\alpha}} \leq 1+ \frac{\alpha}{\lambda}$ and from~\eqref{e:Apriori1} the bound $\abs*{n_1 - z_s} \leq C \sqrt{\cF(n)}$, it follows
 \begin{equation*}
  \sum_{l=l_0}^\infty  l^{2\kappa}  k^l \hat n^\omega_l \leq 2\bra*{z_s + \max\set*{C\sqrt{\cF(n)}, q l_0^{-\gamma}}} \sum_{l=l_0}^\infty l^{\alpha+2\kappa} n_l .
 \end{equation*}
 Now, we use the Hölder inequality to interpolate
 \begin{equation*}
   \sum_{l=l_0}^\infty l^{\alpha+2\kappa} n_l \leq \bra*{\sum_{l=l_0}^\infty l n_l}^{\frac{\alpha + 2\kappa+\gamma-1}{\gamma}} \bra*{\sum_{l=l_0}^\infty l^{1-\gamma} n_l}^{\frac{1-\alpha-2\kappa}{\gamma}} \leq C\; \bra*{\cF^{\LSW}_{l_0}(n(t))}^{\frac{1-\alpha-2\kappa}{\gamma}}
 \end{equation*}
 by assuming $1-\alpha-\gamma\leq 2\kappa \leq 1-\alpha$.
 The estimate~\eqref{e:L1FluxEst} follows from~\eqref{e:L1ActionEst} by noting that with the choice $\phi^*(t) = D \cF(n(t)) = \bra*{\log \frac{n_l(t)}{\omega_l}}_{l\geq 1}$ holds
 \[
    k^l \hat n^\omega_l(t) \abs*{\nabla_l \phi^*(t)} = a_l n_1(t) n_l(t) - b_{l+1} n_{l+1}(t)
 \]
 and $\cA_{\mac}(n(t),\phi^*(t)) = \cD_{\mac}(n(t))$.
\end{proof}
The last a priori estimate deals with tightness and how tightness is preserved for curves of finite action.
\begin{lemma}[Tightness is preserved for curves of finite action]
  A family $N \subset \cM$ is called tight provided that
  \begin{equation}\label{ass:tight-initial}
    \sup_{n\in N} \sum_{l = R}^\infty l n_l \to 0
     \qquad \text{as } R\to \infty .
  \end{equation}
  If the family $N_0 \subset \cM$ satisfy the tightness condition~\eqref{ass:tight-initial}. Then for any $T>0$ and any family of curves $\set{ (n,\phi) \in \cCE_T: n(0) \in N_0}$ of uniformly finite action the family $\set{ n(t)}$ also satisfies the tightness condition~\eqref{ass:tight-initial} for $t\in [0,T]$.
\end{lemma}
\begin{proof}
  The proof is similar to Lemma~\ref{lem:LSW:tight}, where the same result is proven for the LSW gradient structure.
  Let $1\ll M_1 \ll M_2$ and let $\eta \in C^1(\R)$ be a cut off function such that $\eta(l)=0$ for $l\leq \frac{M_1}{2}$ and $l\geq 2M_2$, $\eta(l) = 1$ for $M_1 \leq l \leq M_2$ and such that $\eta'(l)\leq \frac{C}{M_1}$ for $\frac{M_1}{2} \leq l \leq M_1$ and $\abs*{\eta'(l)}\leq \frac{C}{M_2}$ for $M_2\leq l \leq 2M_2$.
  Moreover, we define $\cN_l := l \eta(l)$ and assume $M_1>2$ such that $\eta(1)=0$.
  Then, it follows for any curve of finite action $(n,\phi)\in \Phi$
  \begin{align*}
    \pderiv{}{t} \cN \cdot n(t) &=  \cN \cdot \partial_t n(t) = \cN \cdot \cK(n) \phi \\
    &=\bra*{\sum_{l=1}^\infty \eta(l+1) \, k^l \, \widehat{n}^\omega_l(t)\; \nabla_l \phi(t) + \sum_{l=1}^\infty k^l \, \widehat{n}^\omega_l(t) \, l \, \nabla_l \eta \; \nabla_l \phi(t)} \\
    &\leq C \; \Biggl( \frac{1}{M_1^{\frac{1-\alpha}{2}}} \sum_{l=M_1/2}^\infty l^{\frac{1-\alpha}{2}}  k^l \widehat{n}^\omega_l(t) \abs*{\nabla_l \phi(t)} \\
    &\qquad + \bra*{\frac{1}{M_1} M_1^{1-\frac{1-\alpha}{2}} + \frac{1}{M_2} M_2^{1-\frac{1-\alpha}{2}}} \sum_{l=M_1/2}^\infty  l^{\frac{1-\alpha}{2}}  k^l \widehat{n}^\omega_l(t) \abs*{\nabla_l \phi(t)} \Biggr) \\
    &\leq C \; \bra*{ \frac{1}{M_1^{\frac{1-\alpha}{2}}} + \frac{1}{M_2^{\frac{1-\alpha}{2}}} }  \sum_{l=M_1/2}^\infty l^{\frac{1-\alpha}{2}}  k^l \widehat{n}^\omega_l(t) \abs*{\nabla_l \phi(t)} ,
\end{align*}
where $C$ is the constant depending only on the cut off function $\eta$.
Integrating over time and using~\eqref{e:L1ActionEst} leads to
\begin{align*}
  \sum_{l=M_1}^{M_2} l n_l(t) \leq \sum_{l=M_1/2}^{2M_2} l n_l(0) + C \bra*{\frac{1}{M_1^{\frac{1-\alpha}{2}}} + \frac{1}{M_2^{\frac{1-\alpha}{2}}}} \int_0^t \sqrt{\cA(n(t),\phi(t))} \dx{t} .
\end{align*}
Now, using the fact that $t\mapsto n(t)$ is a curve of finite action and letting $M_2\to \infty$, we obtain
\begin{equation*}
 \sum_{l=M_1}^{\infty} l n_l(t) \leq \sum_{l=M_1/2}^\infty l n_l(0) + \frac{C t^{\frac{1}{2}}}{M_1^{\frac{1-\alpha}{2}}} ,
\end{equation*}
where the constant $C$ is uniform for the family. This finishes the proof since $N_0$ satisfies the tightness condition~\eqref{ass:tight-initial}.
\end{proof}

\subsection{Passage to the limit: Proof of Theorem~\ref{thm:Scale:conv:CfA}}\label{S:Limit:sub}

To pass to the limit in the discrete continuity equation, we define the flux density measure for a fixed covector~$\phi$ and rescaled one~$w^\eps(\eps l) = z_s \eps^{-\gamma} \nabla_l \phi$ (cf.\ \eqref{def:Scale:vectorfield}) by
\begin{align}\label{def:Scale:fluxdensity}
  \mu^\eps(\dx\lambda) &:= \frac{z_s}{\eps^{1-\alpha+2\gamma}} \sum_{l\geq l_0} \delta_{\eps l}(\dx\lambda) \, k^l \,  \widehat{n^\eps}^\omega_l \, \nabla_l \phi \\
  &= \frac{1}{\eps^{1-\alpha+\gamma}} \sum_{l\geq l_0} \delta_{\eps l}(\dx\lambda) \, l^\alpha \,\Lambda\bra*{n_1^\eps n_l^\eps , (z_s + q(l+1)^{-\gamma}) n_{l+1}^\eps} \, w^\eps(\lambda) .\notag
\end{align}
and the dissipation flux density measure
\begin{align}\label{def:Scale:fluxdensityhat}
  \hat\mu^\eps(\dx\lambda) &:= \frac{1}{\eps^{1-\alpha+\gamma}} \sum_{l\geq l_0} \delta_{\eps l}(\dx\lambda) \bra*{ a_l n_1(t) n_l(t) - b_{l+1} n_{l+1}(t)}
\end{align}
Let us note, that with the above definitions for $l\geq l_0$ and $\lambda =\eps l$ holds
\begin{equation}\label{e:Scale:continuityEqu}
  \dot n^\eps_l(t) - \frac{1}{\eps^{1-\alpha+\gamma}}(\cK[n] \phi)_l
  = \partial_t \nu^\eps_t(\lambda) + \partial^\eps_{\lambda} \mu^\eps_t(\lambda) = 0  ,
\end{equation}
where
\begin{equation*}
  \partial_\lambda^\eps \mu_t^\eps(\lambda) := \frac{\mu_t^\eps(\lambda+\eps) - \mu_t^\eps(\lambda)}{\eps} .
\end{equation*}
Let us summarize the a priori estimates found in Section~\ref{s:AprioriBD} and rewrite them in rescaled variables. We denote with $\cF^\eps_ {\mac}(\Pi_{\mac}^\eps n) = \eps^{-\gamma} \cF_{\mac}(n)$  and similarly for $\cA^\eps_{\mac}$ as well as $\cD^\eps_{\mac}$.
\begin{proposition}[Rescaled a priori estimates]
  With $x$ from~\eqref{ass:cutoff} holds
  \begin{enumerate}[ i) ]
    \item The rescaled free energy satisfies
      \begin{equation}\label{e:Scale:energybound}
        C \geq \cF^\eps(n^\eps) \geq \cF^\eps_{\mac}(\nu^\eps) = E(\nu^\eps) \bra*{1+ O(\eps^{x\sigma})}
      \end{equation}
     \item The total excess mass satisfies
       \begin{equation}\label{e:Scale:MassExcess}
         \abs*{\int \lambda \nu^\eps(\dx{\lambda}) - (\varrho_0 - \varrho_s)} \leq C \eps^{\gamma\bra*{\frac{1}{2}-x}} \sqrt{\cF^\eps_{\mac}(\nu^\eps)} .
       \end{equation}
     \item Let $[0,T] \ni t \mapsto \nu_t^\eps \in \cM^\eps$ be a rescaled curve of finite action and $\eta\in L^2((0,T))$, then for any $\kappa \in \pra*{\frac{1-\alpha-\gamma}{2},\frac{1-\alpha}{2}}$
       \begin{align}
         \int_0^T \eta(t) \int \lambda^{\kappa} \abs*{\mu^\eps_t(\dx\lambda)} \dx{t} &\leq C \;\bra*{\sup_{t\in [0,T]} \cF^{\eps}_{\mac}(\nu^\eps_t)}^{\frac{1-\alpha-2\kappa}{2\gamma}} \int_0^T  \abs*{\eta(t)} \sqrt{\cA^\eps_{\mac}(\nu^\eps_t,w^\eps_t)} \dx{t} , \label{e:Scale:ActionMoment}\\
         \int_0^T  \eta(t) \int \lambda^{\kappa} \abs*{\hat\mu^\eps_t(\dx\lambda)} \dx{t}
         &\leq C \;\bra*{\sup_{t\in [0,T]} \cF^{\eps}_{\mac}(\nu^\eps_t)}^{\frac{1-\alpha-2\kappa}{2\gamma}} \int_0^T  \abs*{\eta(t)} \sqrt{\cD^\eps_{\mac}(\nu^\eps_t)} \dx{t} . \label{e:Scale:FluxMoment}
       \end{align}
     \item If $\set{\nu_0^\eps}_{\eps>0}$ satisfies the tightness condition~\eqref{ass:tightness} and $[0,T]\ni t \mapsto \nu_t^\eps \in \cM^\eps$ are rescaled curves of finite action, then for all $t\in [0,T]$ also $\set{\nu_t^\eps}_{\eps>0}$ satisfies the tightness condition~\eqref{ass:tightness}.
  \end{enumerate}
\end{proposition}
The above results enable us to conclude the $\liminf$ estimates and proof Theorem~\ref{thm:Scale:conv:CfA}.
\begin{proof}[Proof of Theorem~\ref{thm:Scale:conv:CfA}]
  \emph{Step 1: Convergence of $\nu^\eps$.}
  For $\zeta \in C_c^1(\R_+)$ and $0\leq t_1 < t_2 \leq T$, we calculate using the discrete continuity equation in the form~\eqref{e:Scale:continuityEqu}
  \begin{align*}
  \abs*{\int \zeta \dx\nu_{t_1}^\eps - \int \zeta \dx\nu_{t_1}^\eps} &= \abs*{ \int_{t_1}^{t_2} \int \partial_\lambda^\eps\zeta(\lambda) \mu^\eps_t(\dx\lambda) \dx{t}} \\
  &\leq \sup_{\lambda\in \R_+} \frac{\abs*{\partial_\lambda^\eps \zeta(\lambda)}}{\lambda^{\frac{1-\alpha}{2}}} \int_{t_1}^{t_2} \int \lambda^{\frac{1-\alpha}{2}} \abs*{\mu^\eps_t(\dx\lambda)} \dx{t} \\
  &\stackrel{\mathclap{\eqref{e:Scale:ActionMoment}}}{\leq} C \sup_{\lambda\in \R_+} \frac{\abs*{\partial_\lambda^\eps \zeta(\lambda)}}{\lambda^{\frac{1-\alpha}{2}}} \int_{t_1}^{t_2} \sqrt{\cA^\eps_{\mac}(\nu^\eps_t, w^\eps_t)} \dx{t} \\
  &\leq C \sup_{\lambda\in \R_+} \frac{\abs*{\zeta'(\lambda)}}{\lambda^{\frac{1-\alpha}{2}}} \sqrt{\abs*{t_1 - t_2}} .
  \end{align*}
  This estimate together with the bound~\eqref{e:Scale:energybound} imply via Arzelà-Ascoli the weak$^*$ convergence towards a weakly$^*$ continuous map $t\mapsto \nu_t$.
  Moreover, the a priori bounds~\eqref{e:Scale:energybound}, \eqref{e:Scale:MassExcess} and tightness condition~\eqref{ass:tightness} imply that $\int \zeta \dx\nu^\eps \to \int \zeta \dx\nu$ holds for $\zeta \in C^0(\R_+)$ satisfying
  \begin{equation*}
    \limsup_{\lambda\to\infty} \frac{\abs*{\zeta(\lambda)}}{\lambda} < \infty \qquad\text{and}\qquad \lim_{\lambda\to 0} \frac{\abs*{\zeta(\lambda)}}{\lambda^{1-\gamma}} = 0  ,
  \end{equation*}
  which implies that the excess mass is preserved
  \begin{equation*}
    \int \lambda \dx\nu_t = \varrho_0 - \varrho_s \quad \Rightarrow \quad \nu_t \in M , \qquad \text{for all } t\in[0,T].
  \end{equation*}
  Moreover, the bounds~\eqref{e:Scale:energybound} and \eqref{e:Scale:MassExcess} also imply by weak lower semi-continuity the estimate~\eqref{e:Scale:conv:Fliminf} and especially that $\sup_{t\in [0,T]} E(\nu_t) < \infty$.

  \emph{Step 2: Convergence of $\mu^\eps$.} The a priori estimate~\eqref{e:Scale:ActionMoment} implies the existence of a measure $\mu \in C_c^0([0,T] \times \R_+)^*$ such that up to subsequences
  \begin{equation}\label{e:Scale:conv:fluxdensity:p1}
    \int \int \zeta(t,\lambda) \dx\mu_t^\eps \dx{t} \to \int \int \zeta(t,\lambda) \dx\mu \qquad \text{for all } \zeta \in C_c^0([0,T] \times \R_+) .
  \end{equation}
  Now, we show the limiting measure is of the form $\mu(\dx t, \dx\lambda) = \lambda^\alpha w_t(\lambda) \nu_t(\dx\lambda) \dx{t}$ for some vector field $w_t$ with finite action.
  Therefore, we remind at the definition of $\mu^\eps$~\eqref{def:Scale:fluxdensity} and $\cA^\eps_{\mac}$~\eqref{def:Scale:action} to estimate
  \begin{align}\label{e:Scale:conv:fluxdensity:p2}
    \int\int \zeta(t,\lambda) \lambda^{\frac{1-\alpha}{2}} \mu_t^\eps(\dx\lambda) \dx{t} &\leq \bra*{\int_0^T \cA^\eps_{\mac}(\nu^\eps_t,w^\eps_t) \dx{t}}^\frac{1}{2} \times \\
    &\qquad\bra*{\int_0^T \frac{1}{z_s} \sum_{l\geq l_0} \zeta^2(t,\eps l) \, l^{1-\alpha} k^l \, \widehat{n^\eps}^\omega_l(t) \dx{t}}^{\frac{1}{2}} . \notag
  \end{align}
  The second term on the right hand side can be bounded by using the one-homogeneity and concavity of $(a,b)\mapsto \Lambda(a,b)$
  \begin{align*}
    \MoveEqLeft{\frac{1}{z_s}\sum_{l\geq l_0} \zeta^2(t,\eps l) \, l^{1-\alpha} k^l \, \widehat{n^\eps}^\omega_l(t) } \\
    &\leq \frac{1}{z_s}\sum_{l\geq l_0} \zeta^2(t,\eps l)  \, l \, \Lambda\bra*{n_1^\eps(t) n_l^\eps(t) , (z_s + q (l+1)^{-\gamma}) n_{l+1}^\eps(t)} \\
    &\leq \frac{1}{z_s}\Lambda\bra*{ n_1^\eps(t) \sum_{l\geq l_0} \zeta^2(t,\eps l) \, l \, n_l^\eps(t) , \bra*{z_s + q l_0^{-\gamma}}\sum_{l\geq l_0} \zeta^2(t,\eps l) \, l \, n_{l+1}^\eps(t)  } \\
    &\leq \frac{z_s + o(1)}{z_s} \sum_{l\geq l_0} \zeta^2(t,\eps l) \, l \, n_l^\eps(t),
  \end{align*}
  where we used in the last estimate that $l_0 =\eps^{-x}$, $|n_1^\eps - z_s| \leq C \sqrt{\cF(n)} \leq C \eps^\frac{\gamma}{2}$ by~\eqref{e:Apriori1} and the fact that $\zeta$ is uniformly continuous.
  Since $t\mapsto n^\eps(t)$ is a curve of finite action and by the convergence of the total mass, it follows that the right hand side of~\eqref{e:Scale:conv:fluxdensity:p2} is finite.
  Hence, we can pass to the limit $\eps\to 0$ in~\eqref{e:Scale:conv:fluxdensity:p2} by the same argument as in~\eqref{e:Scale:conv:fluxdensity:p1}.
  It follows for a subsequence which attains
  \begin{equation*}
     \int_0^T \cA^\eps_{\mac}(\nu^\eps, w^\eps) \dx{t} \to  A^* :=  \liminf_{\eps\to 0} \int_0^T \cA^\eps_{\mac}(\nu^\eps, w^\eps) \dx{t}.
  \end{equation*}
  the estimate
  \begin{equation*}
    \int\int \zeta(t,\lambda) \lambda^\frac{1-\alpha}{2} \mu_t(\dx\lambda) \dx{t} \leq \bra*{ A^* \int\int \zeta^2(t,\lambda) \lambda \nu_t(\dx\lambda)\dx{t}}^\frac{1}{2} ,
  \end{equation*}
  with $\mu_t(\dx\lambda)$ denoting the disintegration of $\mu$ in $t$.
  Hence, we can conclude as in the derivation of~\eqref{e:LSW:compact:Action:p1} to find $v\in L^2(\lambda \dx\nu_t \dx{t})$ by the Riesz representation theorem showing lower semi-continuity of the action~\eqref{e:Scale:conv:Aliminf}.

  \emph{Step 3: Convergence of the dissipation~$\cD^\eps$.} We observe that $\cD^\eps_{\mac}(\nu_t^\eps) = \cA^\eps_{\mac}(\nu_t^\eps,{\tilde w}^\eps)$, where ${\tilde w}^\eps$ is the special vector field given by $-\nabla_l D\cF^\eps(\nu_t^\eps)$, i.e.\ for all~$l$
  \[
    \eps^\gamma {\tilde w}^\eps_t(\eps l) = \log \frac{n_1^\eps n_l^\eps}{\omega_1 \omega_l} - \log \frac{n_{l+1}^\eps}{\omega_{l+1}} .
  \]
  Therefore, we can apply the same arguments of step 2, but now to the dissipation flux density $\hat\mu^\eps$ defined in~\eqref{def:Scale:fluxdensityhat} and use the a priori estimate~\eqref{e:Scale:FluxMoment} to deduce the $\liminf$ estimate
  \[
    \liminf_{\eps\to 0} \int_0^T \cD^\eps(\nu_t^\eps) \dx{t} \geq \int_0^T A(\nu_t, \tilde w_t) \dx{t} =  \int_0^T \int \lambda^\alpha \abs*{\tilde w_t}^2 \nu_t(\dx\lambda) \dx{t},
  \]
  for some $\tilde w \in C_c^0([0,T]\times \R_+)^*$.
  It, is left to show that $h^\eps(t) \rightharpoonup h(t)$ in $L^2((0,T))$ and $\tilde w_t$ is of the form $h(t) -q/\lambda^\gamma$ wit $h\in L^2([0,T])$, however this statement follows exactly along the lines of \cite[Lemma 2.6]{Niethammer2003}.
  The final result~\eqref{e:Scale:conv:Dliminf} follows now by the definition of $D(\nu_t)$ as the infimum over all such $h\in L^2([0,T])$ from Lemma~\ref{lem:LSW:Dissipation}).

  \emph{Step 4: Continuity equation holds.} Finally, choosing a subsequence such that both convergences~\eqref{e:conv:nu} and~\eqref{e:conv:mu} holds for a test function $\zeta \in C_c^\infty([0,T]\times \R)$, we can pass to the limit in the weak form of the discrete continuity equation~\eqref{e:Scale:continuityEqu}
  \begin{align*}
    \int_0^T \int \partial_t \zeta(t,\lambda) \nu_t^\eps(\dx\lambda)\dx{t} &+ \int\int \partial_\lambda^\eps \zeta(t,\lambda) \mu_t^\eps(\dx\lambda) \dx{t} = 0 \\
    \downarrow \ \eps\to 0 \qquad & \qquad\qquad\qquad \downarrow  \eps\to 0 \\
    \int_0^T \int \partial_t \zeta(t,\lambda) \, \nu_t(\dx\lambda)\dx{t} &+ \int\int \partial_\lambda \zeta(t,\lambda) \, \mu_t(\dx\lambda) \dx{t} = 0 ,
  \end{align*}
  which shows~$(\nu,w)\in \CE_T$.
\end{proof}

\subsection{Quasistationary expansion: Proof of Theorem~\ref{thm:QuasiSmall}}\label{S:quasi}

The proofs of Proposition~\ref{prop:stability} and Theorem~\ref{thm:QuasiSmall} consists in several steps, which are formulated in the following Lemmata.
In the proofs of this section, $C$ is a generic constant, which is assumed to be independent of $\eps$ and only depending on the parameters inside of the rates from Assumption~\ref{ass:BD:rates}.
\begin{lemma}\label{lem:LowerMoment}
  Assume that $\cF_{\mac}^\eps(\nu^\eps) \leq C$, $l_0$ satisfies~\eqref{ass:cutoff} and $\nu^\eps$ satisfies the tightness condition~\eqref{ass:tightness}.
  Then for any $\kappa\in [0,1]$, there exists $c>0$ such that
  \begin{equation}\label{e:monexp:lowerMoment}
    \int \lambda^{\kappa} \dx{\nu^\eps} \geq c > 0 \qquad\text{uniformly in } \eps >0.
  \end{equation}
\end{lemma}
\begin{proof}
 The assumptions of the Lemma ensure the conservation of the excess mass~\eqref{e:Scale:MassExcess}.
 Together with the tightness assumption, we have for any $\kappa \leq 1 $
 \begin{align*}
   \int \lambda^{\kappa} \; \nu^\eps_t(\dx\lambda) &\geq \int_0^M \lambda^{\kappa} \; \nu^\eps_t(\dx\lambda) \\
    &\geq \frac{1}{M^{1-\kappa}}  \int_0^M \lambda \; \nu_t^\eps(\dx{\lambda}) \\
    &\geq \frac{1}{M^{1-\kappa}} \bra*{ \rho_0 - \rho_s - O_\eps\bra*{\eps^{\gamma\bra*{1/2 -x}}} - o_M(1) } .
 \end{align*}
 Hence, we can choose $M$ large enough but finite and $\eps$ small enough such that for some $c>0$ the estimate~\eqref{e:monexp:lowerMoment} holds.
\end{proof}
\begin{lemma}
  For any $n\in \cM$ holds
  \begin{equation}\label{e:DissBound0}
    \bra*{ u - h } \log\bra*{\frac{1+u}{1+h}} \leq \frac{\cD_{\mac}(n)}{A(z_s)} ,
  \end{equation}
  where $u:=u(z_s)$,
  \begin{equation*}
    h := \frac{n_1-z_s}{z_s} \qquad\text{and}\qquad u(z) := \frac{B - A(z)}{A(z)} ,
  \end{equation*}
  and
  \begin{equation}\label{e:def:barA:barB}
    A(z) := \sum_{l\geq l_0} a_l z n_l \qquad\text{and}\qquad  B := \sum_{l\geq l_0} b_{l+1} n_{l+1} .
  \end{equation}
  Moreover, if $\cF(n)\leq C \eps^{\gamma}$, $l_0$ satisfies~\eqref{ass:cutoff} and $\nu^\eps$ satisfies the tightness condition~\eqref{ass:tightness}, then it holds for $\eps$ small enough and some $C>0$ uniformly in $\eps$ the estimate~\eqref{e:DissBound} from Proposition~\ref{prop:stability}.
\end{lemma}
\begin{proof}
  For the proof $n$ is fixed such that $\cF(n)\leq C \eps^\gamma$.
  Then, we introduce two measures $\alpha$ and $\beta$ on $\set*{l_0,l_0+1,\dots}$
  \[
     \forall l\geq l_0:\qquad \alpha_l(z) := a_l z n_l \qquad\text{and}\qquad \beta_l := b_{l+1} n_{l+1}
  \]
  with partition sums $A(z)$ and $B$~\eqref{e:def:barA:barB}, respectively.

  We introduce $\cD_{\mac,z}(n)$ the constant monomer density dissipation of the large clusters
  \begin{align*}
    \cD_{\mac,z}(n) := \sum_{l\geq l_0} \bra*{ a_l z n_l - b_{l+1} n_{l+1}} \log\frac{a_l z n_l}{b_{l+1} n_{l+1}} \geq 0.
  \end{align*}
  Note, that by this definition $\cD_{\mac}(n) = \cD_{\mac,n_1}(n)$.
  By the definition~\eqref{e:def:barA:barB}, it follows $A(n_1)=(1+h)A(z_s)$ and the identity
  \begin{equation}\label{e:MacDiss:uhId}
    u(n_1) = \frac{1}{1+h}\bra*{u(z_s) - h }.
  \end{equation}
  Now, rewrite and $\cD_{\mac,z}(n)$ and apply the Jensen inequality to the one-homogeneous convex function $a \mapsto a \log\bra*{1+a}$
  \begin{align*}
    \cD_{\mac,z}(n) &= \sum_{l\geq l_0} \alpha_l(z) \frac{\beta_l-\alpha_l(z)}{\alpha_l(z)} \log\bra*{1+\frac{\beta_l-\alpha_l(z)}{\alpha_l(z)}} \\
    &\geq \bra*{B- A(z)} \log\frac{B}{A(z)} = A(z)\bra*{u(z) \log(1+u(z)}
  \end{align*}
  Hence, we obtain by setting $z=n_1$ and using~\eqref{e:MacDiss:uhId}, the estimate
  \begin{align*}
    A(z)\bra*{u(z) \log(1+u(z)} &= A(z_s) \bra*{u - h} \log\bra*{1+\frac{u-h}{1+h}},
  \end{align*}
  from where we conclude~\eqref{e:DissBound0}.
  By using the explicit expression of the rates~\eqref{ass:BD:rates} follows
  \begin{align*}
    2 A(z_s) - B &= \sum_{l\geq l_0} z_s l^{\alpha}\bra*{1- \frac{q}{z_s l^{\gamma}}}  n_l + z_s l_0^\alpha n_{l_0} + q l_0^{\alpha-\gamma} n_{l_0} \\
    &\geq \bra*{1 - \frac{q}{z_s l_0^{\gamma}}} A(z_s) \geq A(z_s) \bra*{1- O(\eps^{x\gamma})} ,
  \end{align*}
  by the definition of $l_0$~\eqref{ass:cutoff}.
  Hence, we have $ B \leq A(z_s)  \bra*{1+ O(\eps^{x\gamma})}$ and in particular with~\eqref{e:MacDiss:uhId}
  \begin{equation}\label{e:MacDiss:upper_u}
    u(n_1) \leq \frac{u(z_s)}{1-\abs*{h}} + \frac{\abs*{h}}{1-\abs*{h}} \leq O(\eps^{x\gamma}) + O\bra*{\eps^{\frac{\gamma}{2}}} = O(\eps^{x\gamma}),
  \end{equation}
  where we used that with $\cF(n)\leq C \eps^\gamma$ also $\abs*{h}\leq C \eps^\frac{\gamma}{2}$ from the estimate~\eqref{e:Apriori1}.
  The estimate~\eqref{e:MacDiss:upper_u} allows to linearize the bound~\eqref{e:DissBound0} as follows
  \begin{equation*}
    \bra*{u - h}\log\bra*{\frac{1+h}{1+u}} \geq \frac{(u-h)^2}{\max\set*{1+h,1+u}} \geq \frac{(u-h)^2}{1+ O(\eps^{x\gamma})} .
  \end{equation*}
  Finally, to deduce the estimate~\eqref{e:DissBound}, it is enough to rewrite it in rescaled variables and use the estimate~\eqref{e:monexp:lowerMoment} from Lemma~\ref{lem:LowerMoment}
  \begin{equation*}
   A = z_s \eps^{1-\alpha} \int \lambda^{\alpha} \nu^\eps(\dx\lambda) \geq c \eps^{1-\alpha} > 0 . \qedhere
  \end{equation*}
\end{proof}
The time-scale separation between the dynamic of the small clusters and the one of the large clusters is characterized by the following logarithmic Sobolev type inequality.
\begin{proposition}[Microscopic energy-dissipation estimate]
  Let $\omega_l(z) := z^l Q_l$.
  Then for all $n\in \R_+^\N$ with $\cF_{\mic}(n)\leq C\eps^{\gamma}$ there exists $C_{\EED}$ independent of $\eps$ such that it holds
  \begin{equation}\label{e:BD:MLSI}
    \cH_{\mic}(n \mid \omega(n_1)) \leq C_{\EED} \eps^{-x\bra*{1-\alpha+\gamma}} \cD_{\mic}(n) ,
  \end{equation}
  where $\cH_{\mic}$ is the microscopic part of the relative entropy between $n$ and $\omega(n_1)$ defined by
  \begin{equation*}
    \cH_{\mic}(n \mid \omega(z)) := \sum_{l=1}^{l_0-1} \omega_l(z) \psi\bra*{\frac{n_l}{\omega_l(z)}} \qquad\text{with}\qquad \psi(a) = a \log a - a +1 .
  \end{equation*}
\end{proposition}
\begin{proof}
  We note, that the function  $(a,b) \mapsto \varphi(a,b):=(a-b)\bra*{\log a - \log b}$ occurring in the definition of the dissipation is one-homogeneous.
  In addition, the following lower bound holds
  \[
    (a-b)\bra*{\log a - \log b} \geq 4 \bra*{\sqrt{a} - \sqrt{b}}^2 .
  \]
  Moreover, we remind that $\omega(z)$ satisfies the detailed balance condition $a_l \omega_1(z) \omega_l(z) = b_{l+1} \omega_{l+1}(z)$.
  By choosing $z=n_1$, the dissipation can be rewritten and bounded from below by
  \begin{align*}
    \cD(n) &= \sum_{l\geq 1} a_l n_1 \omega_l(n_1) \ \varphi\bra*{\frac{n_l}{\omega_l(n_1)}, \frac{n_{l+1}}{\omega_{l+1}(n_1)}}  \\
    &\geq 4  \sum_{l\geq 1} a_l n_1 \omega_l(n_1) \bra*{\sqrt{\frac{n_l}{\omega_l(n_1)}} - \sqrt{ \frac{n_{l+1}}{\omega_{l+1}(n_1)}}}^2 := \overline \cD(n).
  \end{align*}
  Hence, instead of showing the estimate~\eqref{e:BD:MLSI}, it is sufficient to proof
  \[
    \cH_{\mic}\bra*{n\mid \omega(n_1)} \leq C_{\EED} \overline \cD(n).
  \]
  This inequality was investigated in~\cite{Canizo2015}.
  To apply the result there, we introduce the measures
  \begin{equation*}
    l\in \set*{1,\dots,l_0-1} : \quad \mu_l(z) := \frac{\omega_l(z)}{\sum_{l=1}^{l_0-1} \omega_l(z)} \quad\text{and}\quad  \nu_l(z) := \frac{a_l \omega_l(z)}{\sum_{l=1}^{l_0-1} \omega_l(z)} .
  \end{equation*}
  Hereby, we note that $\mu$ is a probability measure, but $\nu$ not necessarily.
  Let us assume the following mixed logarithmic Sobolev inequality
  \begin{equation}\label{e:BD:LSI}
    \Ent_{\mu}(f^2) := \sum_{l=1}^{l_0-1} \mu_l f_l^2 \log\frac{f_l^2}{\sum_{l=1}^{l_0-1} f_l^2 \mu_l} \leq C_{\LSI} \sum_{l=1}^{l_0-1} \nu_l \bra*{ f_l - f_{l+1}}^2 .
  \end{equation}
  Then, \cite[Proposition 3.2]{Canizo2015}, where by the different normalization of $\nu$, the constant simplifies to
  \begin{equation}\label{e:BD:MLSI:p2}
    C_{\EED}\leq \frac{C_{\LSI}}{n_1^2} \bra*{ n_1^2 + 2 \,\bra*{\sum_{l=1}^{l_0-1} n_l}\,\bra*{\sum_{l=1}^{l_0-1} \omega_l(n_1)}}.
  \end{equation}
  To proof the mixed logarithmic Sobolev inequality~\eqref{e:BD:LSI}, we use \cite[Corollary 2.4 and Remark 2.5]{Canizo2015}, from which we obtain the bound
  \begin{equation}\label{e:BD:MLSI:p1}
    C_{\LSI} \leq 480 \sup_{1< l < l_0} W_l(n_1) \log\bra*{\frac{W_1(n_1)}{W_l(n_1)}} V_l(n_1),
  \end{equation}
  where
  \[
    W_l(z) = \sum_{j=l}^{l_0-1} \omega_j(z) \quad\text{and}\quad V_l(z) = \sum_{j=1}^{l-1} \frac{1}{a_j \omega_j(z)} .
  \]
  We will establish the following estimates for $\abs*{z-z_s}\leq C \eps^{\frac{\gamma}{2}}$ and some $C>1$
  \begin{align}
     \frac{1}{C} \omega_l(z) \leq W_l(z) &\leq C l^\gamma \omega_l(z) \label{e:BD:MLSI1} \\
      V_l(z) &\leq \frac{C l^{\gamma-\alpha}}{\omega_l(z)} \label{e:BD:MLSI2}
  \end{align}
  We postpone the proof of the estimates and first show the final result.
  By a combination of~\eqref{e:BD:MLSI1} and~\eqref{e:BD:MLSI2} with~\eqref{e:BD:MLSI:p1}, we obtain the estimate
  \begin{align*}
    C_{\LSI}\leq C \sup_{1<l<l_0} l^{2\gamma-\alpha} \log\bra*{\frac{C}{\omega_l(z)}}.
  \end{align*}
  By the expansion~\eqref{e:BD:assymptotic_Q} follows
  \begin{equation*}
    \frac{1}{\omega_l(z)} \leq C l^\alpha \exp\bra*{ \frac{q(1-\gamma)}{z_s} l^{1-\gamma} - l \log\frac{z}{z_s} } \leq C \exp\bra*{C l^{1-\gamma}},
  \end{equation*}
  where we used that $l_0\abs*{\log \frac{z}{z_s}}$ is uniformly bounded, because of $\abs*{z-z_s}\leq C\eps^\frac{\gamma}{2}$.
  We obtain the upper bound $C_{\LSI} \leq C \sup_{1 < l < l_0} l^{1-\alpha+\gamma} \leq C l_0^{1-\alpha+\gamma}$. A combination of this bound with~\eqref{e:BD:MLSI:p2} leads to the bound
  \begin{equation*}
     C_{\EED} \leq C \eps^{-x(1-\alpha+\gamma)}\ \frac{n_1^2 + W_1(n_1) \sum_{l=1}^{l_0-1} n_l}{n_1^2}
  \end{equation*}
  The conclusion~\eqref{e:BD:MLSI} follows now from~\eqref{e:BD:MLSI1}, $\abs*{n_1-z_s}\leq C\eps^{\frac{\gamma}{2}}$ and the bound $\sum_{l=1}^{l_0-1} n_l \leq \sum_{l\geq 1} l n_l = \varrho$. To proof the estimates~\eqref{e:BD:MLSI1} and~\eqref{e:BD:MLSI2}, we first observe that by the assumption $\abs*{z-z_s}\leq C \eps^{\frac{\gamma}{2}}$, we have the comparison
  \[
    \frac{\omega_l(z)}{\omega_l(z_s)} = \bra*{\frac{z}{z_s}}^l \leq \bra*{1 + C \eps^{\frac{\gamma}{2}}}^l \leq \exp\bra*{ C \eps^{\frac{\gamma}{2}} l_0} \leq C
  \]
  by the choice of $l_0$~\eqref{ass:cutoff}.
  In the complete analog way, we get $\frac{\omega_l(z)}{\omega_l(z_s)} \geq \frac{1}{C}$.
  Hence, it is enough to show~\eqref{e:BD:MLSI1} and~\eqref{e:BD:MLSI2} for $z=z_s$.

  Therefore, we use the expansion~\eqref{e:BD:assymptotic_Q} from Lemma~\ref{lem:BD:assymptotic_Q} in the form: For some constant $C>1$ and any $l\geq 1$ holds
  \[
    \frac{1}{C l^{\alpha} z_s^{l-1}} \exp\bra*{ - \frac{q}{z_s} (1-\gamma) l^{1-\gamma}} \leq Q_l \leq \frac{C}{l^{\alpha} z_s^{l-1}} \exp\bra*{ - \frac{q}{z_s} (1-\gamma) l^{1-\gamma}}.
  \]
  The estimate~\eqref{e:BD:MLSI1} with $z=z_s$ is now proven
  \begin{align*}
    W_l(z) &\leq \frac{C}{z_s^2} \sum_{j=l}^{l_0-1} \frac{1}{j^\alpha} \exp\bra*{- \frac{q}{z_s} (1-\gamma) j^{1-\gamma}} \leq \frac{C}{l^{\alpha}} \sum_{j=l}^{l_0-1} \exp\bra*{-\frac{q}{z_s} (1-\gamma) j^{1-\gamma}} \\
    &\leq \frac{C}{l^\alpha} \int_{l}^{l_0-1} \exp\bra*{-\frac{q}{z_s} (1-\gamma) v^{1-\gamma}} \dx{v} \leq C l^{\gamma-\alpha} \int_{l^{1-\gamma}}^\infty \exp\bra*{- \frac{q}{z_s} (1-\gamma) v} \dx{v}\\
    &\leq C l^{\gamma} \omega_l .
  \end{align*}
  The estimate~\eqref{e:BD:MLSI2} with $z=z_s$ follows similarly
  \begin{align*}
    V_{l}(z) &\leq \frac{C}{z_s^2} \sum_{j=1}^{l-1} \exp\bra*{\frac{q}{z_s} (1-\gamma) l^{1-\gamma}}
    \leq \frac{C}{z_s^2} \int_{1}^l \exp\bra*{\frac{q}{z_s} (1-\gamma) v^{1-\gamma}} \dx{v} \\
    &\leq C \frac{(1-\gamma)}{z_s^2} l^{\gamma} \int_0^{l^{1-\gamma}} \exp\bra*{\frac{q}{z_s} (1-\gamma) v} \dx{v}
    \leq C \frac{l^{\gamma-\alpha}}{\omega_l} . \qedhere
  \end{align*}
\end{proof}
The estimates~\eqref{e:monexp:q3} and~\eqref{e:monexp:micF} from Theorem~\ref{thm:QuasiSmall} are a consequence of~\eqref{e:BD:MLSI}.
\begin{proof}[Proof of Theorem~\ref{thm:QuasiSmall}]
  Finally, the estimate~\eqref{e:monexp:q3} follows just by rescaling $\cD_{\mic}$ and integrating the estimate along a curve of finite action.
  For the statement \eqref{e:monexp:micF}, we first observe that $\cF_{\mic}(n) = \cH_{\mic}(n\mid \omega)$ and get by writing $z=z_s e^{h}$
  \begin{align*}
    \cH_{\mic}(n \mid \omega(z) ) - \cF_{\mic}(n) &= - \sum_{l=1}^{l_0-1} l n_l \log\frac{z}{z_s} - \sum_{l=1}^{l_0-1} \omega_l \bra*{1-\bra*{\frac{z}{z_s}}^l} \\
    &= h\; \bra*{\varrho_s - \sum_{l=1}^{l_0-1} l  n_l   + \sum_{l=1}^{l_0-1} l \omega_l \frac{e^{l h}-1}{lh} - \varrho_s }
  \end{align*}
  The first difference in the bracket can be bounded in terms of~\eqref{e:Apriori3} from Lemma~\ref{lem:Apriori}.
  The second difference can be explicitly expresses as follows
  \begin{align*}
    \abs*{\varrho_s -  \sum_{l=1}^{l_0-1} l \omega_l \frac{e^{l h}-1}{lh}} &\leq \sum_{l\geq l_0 } l \omega_l + \sum_{l=1}^{l_0-1} l\omega_l \frac{e^{l h}-1-lh}{lh} \\
    &\leq C l_0^\gamma \omega_{l_0} + C h \sum_{l=1}^{l_0-1} l^2 \omega_l \leq C \bra*{l_0^\gamma \omega_{l_0} + h },
  \end{align*}
  where, we used the upper in~\eqref{e:BD:MLSI1}, which holds by the proof also with $l_0 = \infty$.
  Moreover, $\omega_l$ has arbitrary high moments following from the expansion~\eqref{e:BD:assymptotic_Q}.
  By the choice of $l_0$~\eqref{ass:cutoff} and again~\eqref{e:BD:assymptotic_Q} follows that $l_0^\gamma \omega_{l_0}\leq C \eps^p$  for any $p>0$.
  Hence, combining all these estimates and reminding that $\cF_{\mic}(n)\leq C \eps^\gamma$, we get
  \begin{equation*}
    \abs*{\cH_{\mic}(n \mid \omega(z) ) - \cF_{\mic}(n)} \leq C \abs*{h} \bra*{ \eps^{-x\gamma} \sqrt{\cF_{\mic}(n)} + \eps^p + \abs*{h}} \leq C \abs*{h} \bra*{ \eps^{\sigma} + \abs*{h}} ,
  \end{equation*}
  where $\sigma = \gamma\bra*{\tfrac{1}{2}-x} > 0$ by~\eqref{ass:cutoff}.
  Hence, we can conclude for a curve of finite action
  \begin{align*}
    \int_0^T \!\!\cF_{\mic}^\eps(n(t)) \dx{t} &\leq z_s \eps^{-\gamma} \int_0^T \cH_{\mic}\bra*{n(t) \mid \omega(n_1(t))} \dx{t} \\
    &\quad+ C \eps^{-\gamma} \int_0^T \abs*{\log\bra*{\frac{n_1(t)}{z_s}-1}}\bra*{ \eps^\sigma + \abs*{\log\bra*{\frac{n_1(t)}{z_s}-1}}} \dx{t} \\
    &\leq C\eps^{(1-x)(1-\alpha+\gamma)} \int_0^T \cD^\eps_{\mic}(n(t)) \dx{t} \\
    &\quad + C \eps^{\gamma} \int_0^T \abs*{\frac{n_1(t)-z_s}{\eps^\gamma}}^2 \dx{t} + C \eps^{\sigma} \sqrt{T} \bra*{\int_0^T \abs*{\frac{n_1(t)-z_s}{\eps^\gamma}}^2 \dx{t}}^\frac12 .
  \end{align*}
  The conclusion~\eqref{e:monexp:micF} follows by~\eqref{e:Scale:conv:monomers} from Theorem~\ref{thm:Scale:conv:CfA} .
\end{proof}

\renewcommand\appendixname{}
\appendix

\section{Gradient structures for coagulation and fragmentation models}\label{s:GScfModels}

\subsection{Reversible chemical reactions as gradient flows}

This part of the appendix shows the general structure for reversible chemical reactions. Since, the Becker--Döring equation and other coagulation-fragmentation models can be interpreted as an infinite set of chemical reactions~\eqref{e:BD:ChemReact}, they fall into this category. The basic observation goes back to Mielke~\cite{Mielke2011a}, who found the entropic gradient flow structure for reversible chemical reactions.
\begin{definition}[Reversible chemical reaction]\label{def:revChemReact}
Let $n\in \R_+^N$ be the densities of $N\in \N \cup\set*{+\infty}$ different chemical species (or complexes) $X_i$ reacting according to the mass action law.
Each reaction $r=1,\dots R$ with $R\in \N\cup\set*{+\infty}$ is characterized by the stoichiometric coefficients $x^r,y^r \in \N_0^{N}$ and forward and backward reaction rates $k_\pm^r >0$
\begin{equation}\label{e:GenCR}
 x_1^r X_1 + \dots + x_N^r X_N \stackrel[k_-^r]{k_+^r}{\rightleftharpoons} y_1^r X_1 + \dots + y_N^r X_N , \qquad r=1,\dots , R .
\end{equation}
The chemical reaction is assumed to be reversible. That is, there exists a state $\omega = (\omega_1,\dots,\omega_N)\in \R^N$ such that
\begin{equation}\label{e:DBrel}
 k_+^r \omega^{x^r} = k_-^r \omega^{y^r} =: k^r .
\end{equation}
Here, the notation for multiindices is used: $\omega^{x^r} = \prod_{i=1}^N \omega_i^{x_i^r}$.
The evolution equation for the density is given by
\begin{equation}\label{e:def:ChemRateEvo}
  \dot n = - \sum_{r=1}^R k^r \bra*{\frac{n^{x^r}}{\omega^{x^r}} - \frac{n^{y^r}}{\omega^{y^r}}} \bra*{x^r - y^r} .
\end{equation}
\end{definition}
The Becker--Döring clustering equation interpreted as an infinite set of chemical reactions~\eqref{e:BD:ChemReact} fall in this framework by setting $N=R=\infty$ and $x^r_i := \delta_{i,1}+\delta_{i,r}$ and $y^r_i := \delta_{i,r+1}$.
The detailed balance condition~\eqref{e:DBrel} is satisfied in terms of the one-parameter family of equilibrium distributions~$\omega(z)$~\eqref{e:BD:equilibrium}.
Moreover, the more general Smoluchowski coagulation and fragmentation model fit into this framework (cf.\ Appendix~\ref{s:ex:Smoluchowski}) under the assumption of detailed balance.

The free energy is defined as relative entropy with respect to the reversible equilibrium as in~\eqref{e:def:Lyapunov}, i.e. $\cF(n) = \cH(n\mid \omega)$
and hence
\(
 D\cF(n) = \bra*{\log \frac{n_1}{\omega_1},\dots , \log \frac{n_N}{\omega_N}} .
\)
To define the manifold of states, the stoichiometric subspace and its complement are used
\begin{equation*}
 \cS := \Span\set*{x^r - y^r : r=1,\dots R} \quad\text{and}\quad \cS^\perp :=\set*{s\in \R^N : \gamma \cdot s = 0, \ \forall \gamma \in \cS } .
\end{equation*}
Then, the manifold is given for some fixed $n_0 \in \R_+^N$ by the affine space of densities
\begin{equation*}
 \cM_{n_0} :=  \bra*{n_0 + \cS} \cap \R_+^N = \set*{n\in \R_+^N: n\cdot s = n_0 \cdot s , \forall s\in \cS^\perp}
\end{equation*}
The definition formalizes that $\cS^\perp$ contains all conversation laws of the reaction and therefore the tangent vectors on $\cM_{n_0}$ are given by $\R^\cS$.
Coagulation and fragmentation models of one species, like Becker--Döring, in this terminology are characterized by
\begin{equation*}
  \cS^\perp = \Span\set*{ \mathbf{I} } , \qquad\text{with} \qquad \mathbf{I}  := (1,2,3,4,\dots).
\end{equation*}
Hence, the manifold has only one conserved quantity, which is the density $\varrho_0>0$ of the total number of particles $\cM := \set*{ n \in \R_+^\N : n \cdot \mathbf{I} = \sum_{l=1}^\infty l n_l = \varrho_0 }$.

The derivative of the energy $D\cF$ is a force and has to be interpreted as covector.
The underlying metric can be specified by mapping covectors to (tangent-)vectors.
This is done via the Onsager matrix to be defined as the symmetric semi-positive definite matrix
\begin{equation}\label{e:def:Onsager}
 \cK(n) := \sum_{r} k^r \Lambda\bra*{\frac{n^{x^r}}{\omega^{x^r}},\frac{n^{y^r}}{\omega^{y^r}}} \ (x^r - y^r) \otimes (x^r - y^r)  ,
\end{equation}
where $\Lambda(\cdot,\cdot)$ is the logarithmic mean in~\eqref{e:def:LogMean}.
Hence, recalling that the space of vectors was given by $\R^\cS$, we define the covectors with the help of the Onsager operator by~\eqref{e:def:cotangent}, where the identification is
well-defined since the image of $\cK$ is by definition $\R^\cS$, whenever $n$ is strictly positive in all of its components. Note, although the tangent space is state independent, this is not the case for the cotangent space.

With this preliminary definitions a reversible chemical reaction as given in Definition~\ref{def:revChemReact} is formally the gradient flow of the free energy $\cF$ with respect to the metric structure induced by the Onsager operator \eqref{e:def:Onsager} and it holds the formal identity
\begin{equation}\label{e:GF:ChemRateEvo}
 \dot n = - \cK(n) D\cF(n) .
\end{equation}
The property from which immediately follows that~\eqref{e:GF:ChemRateEvo} is the same as~\eqref{e:def:ChemRateEvo} is
\begin{equation*}
 (x^r - y^r) \cdot D\cF(n) = \sum_{i=1}^n x^r_i \log \frac{n_i}{\omega_i} - y^r_i \log \frac{n_i}{\omega_i} = \log \frac{n^{x^r}}{\omega^{x^r}} -  \log \frac{n^{y^r}}{\omega^{y^r}} ,
\end{equation*}
which is nothing else than the nominator of the logarithmic mean $\Lambda\bra[\big]{\frac{n^{x^r}}{\omega^{x^r}}, \frac{n^{y^r}}{\omega^{y^r}}}$ and resembles a discrete chain rule.
The gradient flow decreases its energy along its evolution in terms of the dissipation, i.e.
\begin{equation*}
\begin{split}
 \pderiv{}{t} \cF(n) &= D\cF(n) \cdot \dot n = - D\cF(n) \cdot \cK(n) D\cF(n) \\
 &= - \sum_r k^r \bra*{\frac{n^{x^r}}{\omega^{x^r}} - \frac{n^{y^r}}{\omega^{y^r}}}   \bra*{\log\frac{n^{x^r}}{\omega^{x^r}} - \log\frac{n^{y^r}}{\omega^{y^r}}} =: - \cD(n)
\end{split}
\end{equation*}
We see that the Becker--Döring system fits into this framework.
However, there is freedom in the choice of the free energy and under certain physical assumption, there are other possible choices.

\subsection{Smoluchowski coagulation and fragmentation equation} \label{s:ex:Smoluchowski}

The Becker--Döring clustering equation is itself just a special case in the more general class of Smoluchowski coagulation and fragmentation equations seen as the following family of chemical reactions
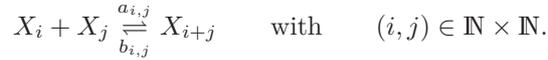
\begin{equation*}
  X_i + X_j \stackrel[b_{i,j}]{a_{i,j}}{\rightleftharpoons} X_{i+j} \qquad\text{with}\qquad (i,j)\in \N\times \N.
\end{equation*}
Hence, the stoichiometric coefficients in~\eqref{e:GenCR} are given as $x^{(i,j)}_k = \delta_{i,k} + \delta_{j,k}$ and $y^{(i,j)}_k = \delta_{i+j,k}$.
A gradient flow structure can be established under the assumption of detailed balance, which in this case does not necessarily hold: There exists a state $\omega\in \R^\N$ such that for $(i,j)\in \N \times \N$ holds
\begin{equation*}
  a_{i,j} \omega_i \omega_j = b_{i,j} \omega_{i+j} .
\end{equation*}
Under this condition,  the Smoluchowski coagulation and fragmentation equation is the gradient flow~\eqref{e:GF:ChemRateEvo} of the free energy $\cF$ with repesct to the Onsager operator $\cK$ defined in~\eqref{e:def:Onsager}.

\subsection{Modified Becker--Döring system}

The modified Becker--Döring system was introduced by Dreyer and Duderstadt~\cite{Dreyer2006a}.
The main feature is the introduction of a mixing entropy between the clusters.
Hence, the free energy consists of a relative entropy part as defined in~\eqref{e:def:Lyapunov} plus a mixing entropy depending on the total number of clusters
\begin{equation}\label{e:FEmodBD}
 \tilde\cF(n) = \cH(n\mid \omega) - N(n) \bra*{\log N(n) - 1} \quad\text{with}\quad N(n) = \sum_i n_i .
\end{equation}
The most compact form of the free energy is $\tilde\cF(n) = \sum_i \bra*{n_i \log \frac{n_i}{\omega_i N(n)} + \omega_i}$.
Hence, the differential of the free energy differential is given by
\begin{equation}\label{e:modBD:GradF}
 D\tilde\cF(n) =\bra*{\log \frac{n_1}{\omega_1 N(n)} , \dots, \log \frac{n_i}{\omega_i N(n)}, \dots}.
\end{equation}
The reaction is still of the same form as the classical Becker--Döring system~\eqref{e:BD:ChemReact}, i.e.\ $x_i^r = \delta_{i,1}+\delta_{i,r}$ and $y_i^r = \delta_{i,r+1}$ in~\eqref{e:GenCR}. This leads to the same detailed balance condition as for the classical Becker-Döring model $a_r \omega_1 \omega_r = b_{r+1} \omega_{r+1} =: k^r$. Hence, we obtain the same possible equilibrium states $\omega_r(z)=z^r Q_r$ given in~\eqref{e:BD:equilibrium}.
Again $z$ has to be determined from the formal conservation law $\sum_{l=1}^\infty l \omega_l(z) = \sum_{l=1}^\infty l n_l$. However, the existence as minimizer of the free energy in this case is more involved and for a detailed analysis of the equilibrium states, we refer to~\cite{Herrmann2005}.

Now, from~\eqref{e:modBD:GradF}, we further deduce
\begin{equation*}
 (x^r - y^r) \cdot D\tilde\cF(n) = \log \frac{n_1 n_r}{\omega_1 \omega_r N(n)^2} - \log \frac{n_{r+1}}{\omega_{r+1}N(n)} =\log \frac{n_1 n_r}{\omega_1 \omega_r} - \log \frac{N(n) n_{r+1}}{\omega_{r+1}} .
\end{equation*}
From the above identity, the modified Onsager matrix can be read off and is given by
\begin{equation*}
 \cK(n) :=  \sum_{r} k^r \Lambda\bra*{\frac{n_1 n_r}{\omega_1 \omega_r} ,  \frac{N(n) n_{r+1}}{\omega_{r+1}}} (x^r - y^r)\otimes (x^r - y^r) .
\end{equation*}
Then, we obtain the modified Becker--Döring equation as the gradient flow of the modified free energy~\eqref{e:FEmodBD}
\begin{equation*}
\begin{split}
 \dot n &= - \tilde \cK(n)D\tilde \cF(n) = -\sum_r k^r \bra*{\frac{n_1 n_r}{\omega_1 \omega_r} - \frac{N(n) n_{r+1}}{\omega_{r+1}}} (x^r - y^r) \\
 &= -\sum_r \bra*{ a_r n_1 n_r - b_{r+1} N(n) n_{r+1}} (x^r - y^r) = -\sum_{r} \tilde J_r (x^r - y^r) .
\end{split}
\end{equation*}
The explicit form of the equation is given for any $l=1,2,\dots$ by
\begin{equation*}
 \dot n_l = \tilde J_{l-1} - \tilde J_l \quad\text{with}\ \ \tilde J_0 := -\sum_{r=1}^\infty \tilde J_r \ \ \text{and}\ \  \tilde J_r(n) := a_r n_1 n_r - b_{r+1}N(n) n_{r+1}.
\end{equation*}

\section{Proof of Lemma~\ref{lem:BD:assymptotic_Q}}\label{s:assymptotic_Q}

\begin{proof}[Proof of Lemma~\ref{lem:BD:assymptotic_Q}]
  We calculate using the definition~\eqref{e:BD:equilibrium} of $Q_l$
  \begin{equation*}
    \log\bra*{l^\alpha z_s^{l-1} Q_l} = - \sum_{j=2}^l \log\bra*{1+ \frac{q}{z_s j^\gamma}}
  \end{equation*}
  The function $x \mapsto \log\bra*{1+ \frac{q}{z_s k^\gamma}}$ is positive, continuous and monotone decreasing to $0$.
  Therefore, we can define the Euler number
  \begin{equation*}
    C_1 := \lim_{l\to \infty} \bra*{\sum_{j=2}^l \log\bra*{1+ \frac{q}{z_s j^\gamma}} - \int_2^l \log\bra*{1+ \frac{q}{z_s x^\gamma}}} .
  \end{equation*}
  Moreover, we get from the Euler-MacLaurin formula the estimate
  \begin{equation*}
    \abs*{C_1 - \bra*{\sum_{j=2}^{l} \log\bra*{1+ \frac{q}{z_s j^\gamma}} - \int_2^{l} \log\bra*{1+ \frac{q}{z_s x^\gamma}}}} \leq \log\bra*{1+\frac{q}{z_s l^\gamma}} \leq \frac{q}{z_s l^\gamma} .
  \end{equation*}
  The following bound
  \begin{equation*}
    \frac{q}{z_s x^\gamma} - \frac{1}{2}\bra*{\frac{q}{z_s x^\gamma}}^2 \leq \log\bra*{1+\frac{q}{z_s x^\gamma}}\leq \frac{q}{z_s x^\gamma} - \frac{1}{2}\bra*{\frac{q}{z_s x^\gamma}}^2 + \frac{1}{3}\bra*{\frac{q}{z_s x^\gamma}}^3
  \end{equation*}
  implies the estimate
  \begin{align*}
   0\leq \frac{\int_{2}^{l} \log\bra*{1+\frac{q}{z_s x^\gamma}} \dx{x} }{\int_{2}^{l}  \bra*{ \frac{q}{z_s x^\gamma} - \frac{1}{2}\bra*{\frac{q}{z_s x^\gamma}}^2  } \dx{x}} - 1
   \leq O(l^{-\gamma}) ,
  \end{align*}
  hereby, we use the convention that $\frac{l^\kappa-1}{\kappa} = \log l$ for $\kappa=0$.
  Now, we can combine all the estimates to obtain
  \begin{align*}
    \log\bra*{l^\alpha z_s^{l-1} Q_l} &= - \bra*{\sum_{j=2}^l \log\bra*{1+\frac{q}{z_s j^\gamma}} - \int_2^l \log\bra*{1+\frac{q}{z_s x^\gamma}} \dx{x}} \\
    &\hspace{-2cm} - \bra*{1+  \frac{\int_{2}^{l} \log\bra*{1+\frac{q}{z_s x^\gamma}} \dx{x} }{\int_{2}^{l}  \bra*{ \frac{q}{z_s x^\gamma} - \frac{1}{2}\bra*{\frac{q}{z_s x^\gamma}}^2  } \dx{x}} -1} \int_{2}^{l}  \bra*{ \frac{q}{z_s x^\gamma} - \frac{1}{2}\bra*{\frac{q}{z_s x^\gamma}}^2  } \dx{x} \\
    &= \bra*{\cF_0 - \frac{q}{z_s (1-\gamma)} l^{1-\gamma} + \frac{q^2}{2 z_s^2(1-2\gamma)} l^{1-2\gamma}} \bra*{1+O(l^{-\gamma})} ,
  \end{align*}
  which concludes the proof by setting $\cF_0=\frac{q 2^{1-\gamma}}{z_s (1-\gamma)} - \frac{q^2 2^{1-2\gamma}}{2 z_s^2(1-2\gamma)} - C_1$.
\end{proof}

\addtocontents{toc}{\SkipTocEntry}
\subsection*{Acknowledgement}

The author wishes to thank Matthias Erbar, Stefan Luckhaus, Babara Niethammer and Juan Velázquez for many fruitful discussions on the Becker-Döring system, LSW equation, gradient flows and related topics.
The author thanks the referees whose incisive and detailed comments have substantially improved the final version of the manuscript.
The author gratefully acknowledges support by the German Research Foundation through the
Collaborative Research Center 1060 \emph{The Mathematics of Emergent Effects}.
Part of this work was done while the author was
enjoying the hospitality of the Hausdorff Research Institute for
Mathematics during the Junior Trimester Program on Optimal Transport.

\printbibliography

\end{document}